\documentclass[11pt,twoside]{article}

\usepackage{geometry}
\usepackage{comment}
\usepackage{epsf}
\usepackage{graphicx}
\usepackage{subfigure}
\usepackage{bbm}
\usepackage{color}
\usepackage{nicefrac}
\usepackage{mathtools}
\usepackage{amsfonts}
\usepackage{amsmath, bm}
\usepackage{amssymb}
\usepackage{textcomp}
\usepackage{xcolor}
\usepackage{appendix}
\usepackage{hyperref} 

\usepackage{algorithm}
\usepackage{algpseudocode} 

\newenvironment{fminipage}%
  {\begin{Sbox}\begin{minipage}}%
  {\end{minipage}\end{Sbox}\fbox{\TheSbox}}

\newcommand{\BlackBox}{\rule{1.5ex}{1.5ex}}  
\ifdefined\proof
    \renewenvironment{proof}{\par\noindent{\bf Proof\ }}{\hfill\BlackBox\\[2mm]}
\else
    \newenvironment{proof}{\par\noindent{\bf Proof\ }}{\hfill\BlackBox\\[2mm]}
\fi
\newtheorem{theorem}{Theorem}

\newtheorem{assumption}{Assumption}
\newtheorem{proposition}{Proposition}
\newtheorem{lemma}{Lemma}

\newtheorem{remark}{Remark}

\usepackage{natbib}
\setcitestyle{round}

\newlength{\widebarargwidth}
\newlength{\widebarargheight}
\newlength{\widebarargdepth}

\makeatletter
\long\def\@makecaption#1#2{
       \vskip 0.8ex
       \setbox\@tempboxa\hbox{\small {\bf #1:} #2}
       \parindent 1.5em 
       \dimen0=\hsize
       \advance\dimen0 by -3em
       \ifdim \wd\@tempboxa >\dimen0
               \hbox to \hsize{
                       \parindent 0em
                       \hfil 
                       \parbox{\dimen0}{\def\baselinestretch{0.96}\small
                               {\bf #1.} #2
                               } 
                       \hfil}
       \else \hbox to \hsize{\hfil \box\@tempboxa \hfil}
       \fi
       }
\makeatother

\long\def\comment#1{}

\DeclareMathOperator*{\argmin}{argmin}
\DeclareMathOperator*{\argmax}{argmax}

\newcommand{\obj}{{\tt obj}} 
\newcommand{\opt}{{\tt opt}}

\newcommand{\supp}{{\tt supp}}
\newcommand{\proj}{{\tt Proj}}

\newcommand{\ub}{{\tt ub}}
\newcommand{\lb}{{\tt lb}}

\newcommand{\ang}{{\tt ang}} 
\newcommand{\strong}{{\tt s}}
\newcommand{\weak}{{\tt w}}
\newcommand{\rate}{{\tt rate}}
\newcommand{\gap}{{\tt gap}}

\newcommand{\spca}{{\tt spca}}
\newcommand{\MIP}{{\tt MIP}}
\newcommand{\MIPr}{{\tt MIP-r}}

\newcommand{\PPM}{{\tt PPM}}
\newcommand{\PLU}{{\tt PLU}}

\begin{document}

\begin{center}

  {\bf{\LARGE{Sparse Principal Component Analysis with Non-Oblivious Adversarial Perturbations}}}
  
\vspace*{.2in}

{\large{
\begin{tabular}{ccc}
Yuqing He$^{\ddagger}$, Guanyi Wang$^{\star}$, Yu Yang$^{\dagger}$
\end{tabular}
}}

\vspace*{.2in}

\begin{tabular}{c}
School of Mathematics and Applied Mathematics, Zhili College$^{\ddagger}$, Tsinghua University \\
Department of Industrial Systems Engineering and Management$^\star$, National University of Singapore \\
Department of Industrial and Systems Engineering$^\dagger$, University of Florida
\end{tabular}

\vspace*{.2in}
\today
\vspace*{.2in}

\begin{abstract}
Sparse Principal Component Analysis (sparse PCA) is a fundamental dimension-reduction tool that enhances interpretability in various high-dimensional settings. An important variant of sparse PCA studies the scenario when samples are adversarially perturbed. Notably, most existing statistical studies on this variant focus on recovering the ground truth and verifying the robustness of classical algorithms when the given samples are corrupted under oblivious adversarial perturbations. In contrast, this paper aims to find a robust sparse principal component that maximizes the variance of the given samples corrupted by non-oblivious adversarial perturbations, say sparse PCA with Non-Oblivious Adversarial Perturbations (sparse PCA-NOAP). Specifically, we introduce a general formulation for the proposed sparse PCA-NOAP. We then derive Mixed-Integer Programming (MIP) reformulations to upper bound it with provable worst-case guarantees when adversarial perturbations are controlled by two typical norms, i.e., $\ell_{2 \rightarrow \infty}$-norm (sample-wise $\ell_2$-norm perturbation) and $\ell_{1 \rightarrow 2}$-norm (feature-wise $\ell_2$-norm perturbation). Moreover, when samples are drawn from the spiked Wishart model, we show that the proposed MIP reformulations ensure vector recovery properties under a more general parameter region compared with existing results. Numerical simulations are also provided to validate the theoretical findings and demonstrate the accuracy of the proposed formulations. 


\end{abstract}
\end{center}

\section{Introduction}
This paper studies a robust generalization of vanilla sparse Principal Component Analysis (PCA) by incorporating additional non-oblivious adversarial perturbations into a given sample set. Such a problem has been widely studied in many fields, including robust training processes with adversarial perturbations \citep{bai2021recent}, machine learning security \citep{diakonikolas2019sever}, exploratory analysis of corrupted data \citep{diakonikolas2017being}, and minimax game theoretic strategy \citep{zhou2019survey}, to name but a few. In particular, the proposed \emph{sparse PCA with Non-Oblivious Adversarial Perturbations (sparse PCA-NOAP)} is defined as follows: Given a set of (unpolluted/unperturbed) samples $\bm{x}_1, \ldots, \bm{x}_n$, for a pre-determined matrix norm $\|\cdot\|$ that specifies the type/metric of adversarial perturbations, and a positive parameter $\rho > 0$ that controls the maximum magnitude of  perturbations, the sparse PCA-NOAP aims to find a \emph{robust \& sparse principal component $\bm{v}^{\tt RS}$} that maximizes the variance under additional non-oblivious adversarial perturbations by solving the following max-min problem: 
\begin{align}
    \bm{v}^{\tt RS} ~ := ~ \argmax_{\bm{v} \in \mathcal{V}_k} ~~  \min_{\widetilde{\bm{\Sigma}} \in  \mathcal{U}_{\|\cdot\|} (\rho)} ~~ \bm{v}^{\top} \widetilde{\bm{\Sigma}} \bm{v} ~~,   \label{eq:SPCA-NOAP} 
\end{align}
where we use $\mathcal{V}_k := \left\{ \bm{v} ~|~ \|\bm{v}\|_2 = 1, \|\bm{v}\|_0 \leq k \right\}$ with a $\ell_0$-norm sparsity constraint for the number of non-zero components of $\bm{v}$ to denote the set of $k$-sparse unit vector, $\widetilde{\bm{\Sigma}}$ to denote an adversarially perturbed (corrupted) covariance matrix choosing from the following uncertainty set $\mathcal{U}_{\|\cdot\|} (\rho) := \{ \widetilde{\bm{\Sigma}} = \frac{1}{n} \widetilde{\bm{X}}^{\top} \widetilde{\bm{X}} ~|~ \widetilde{\bm{X}} = \bm{X} + \bm{E}, ~ \|\bm{E}\| \leq \rho \}$ with the (uncorrupted) sample matrix $\bm{X} := \left( \bm{x}_1 ~|~ \cdots ~|~ \bm{x}_n \right)^{\top} \in \mathbb{R}^{n \times d}$, and perturbation matrix $\bm{E} \in \mathbb{R}^{n \times d}$ controlled by the aforementioned matrix norm $\|\cdot\|$ and parameter $\rho > 0$.

Before presenting the main contributions, we first highlight the differences between the proposed sparse PCA-NOAP and relevant problems arising from data analytics and statistics. First, unlike classical sparse PCA, the proposed sparse PCA-NOAP aims to find a principal component $\bm{v}^{\tt RS}$ that maximizes the variance under adversarial perturbations as presented in Formulation~\eqref{eq:SPCA-NOAP}. Second, compared with the sparse PCA under oblivious outliers arising from statistics, instead of recovering the ground truth or verifying the level of robustness for classical sparse PCA algorithms, sparse PCA-NOAP aims to answer the following two high-level questions: 
\textbf{(Q1)} \emph{Computationally, does there exist a tractable mixed-integer programming (MIP) formulation that enables computing, approximating, and verifying $\bm{v}^{\tt RS}$ on some typical types of non-oblivious adversarial perturbations with provable guarantees?}
\textbf{(Q2)} \emph{Statistically, under different (typical) types of non-oblivious adversarial perturbations, how $\bm{v}^{\tt RS}$ changes with respect to the increasing intensity of adversarial perturbations?} The answer to question \textbf{(Q2)} further sheds light on questions: Which support is more ``significant'' than others? How can we approach sparse PCA with unknown sparsity level $k$? 

\textbf{Contributions and paper organization:}
In general, our main contributions provide an initial affirmative answer to the above two high-level questions, which include the following three parts: \textbf{1.} Section~\ref{sec:SPCA-NOAP-refor} proposes computationally tractable MIPs for two typical sparse PCA-NOAP, i.e., sample-wise/feature-wise adversarially perturbed sparse PCA, combined with provable additive/affine approximation (upper) bounds. For people with independent interests, techniques used in establishing MIP formulations can be generalized to formulate other sparse generalized eigenvalue problems or mixed-integer quadratic programs. \textbf{2.} When samples are i.i.d. generated from some underlying spiked Wishart model, Section~\ref{sec:stat-results} then provides statistical properties for the aforementioned two typical sparse PCA-NOAP, i.e., sample-wise and feature-wise adversarially perturbed sparse PCAs. Additionally, we characterize the behavior of robust \& sparse principal components as parameter $\rho$ increases. \textbf{3.} Numerical simulations are also reported in Section~\ref{sec:numerical} to validate the theoretical findings and demonstrate the relative accuracy of our proposed MIP formulations. Finally, concluding remarks and future directions are included in Section~\ref{sec:summary}. 

\textbf{Notation:} We use lowercase letters, e.g., $a$, for scalars and bold lowercase letters, e.g., $\bm{a}$, as vectors, where $\bm{a}_i$ is its $i$-th component with $i \in [d]$, and bold upper case letters, e.g., $\bm{A}$, as matrices. Without specific description, for a $m$-by-$n$ matrix $\bm{A}$, we denote $\bm{A}_{i,j}$ as its $(i,j)$-th component, $\bm{A}_{i,:}^{\top}$ as its $i$-th row, $\bm{A}_{:,j}$ as its $j$-th column. For a symmetric square matrix $\bm{A}$, we denote $\lambda_{\max}(\bm{A}), \lambda_{\min}(\bm{A})$ and $\lambda_i(\bm{A})$ as its maximum, minimum, and $i$-th largest eigenvalue, respectively. Given index set $\mathcal{I}$, we use $[\bm{a}]_{\mathcal{I}}$ or $\bm{a}_{\mathcal{I}}$ to denote a vector with same values as $\bm{a}$ in $\mathcal{I}$ and rest be zero; $\bm{A}_{\mathcal{I}, \mathcal{I}}$ to denote the submatrix of $\bm{A}$ indexed by $\mathcal{I}$. We denote $\|\bm{a}\|_1, \|\bm{a}\|_2, \|\bm{a}\|_{\infty}, \|\bm{A}\|_F, \|\bm{A}\|_{\text{op}}$ as the $\ell_1, \ell_2, \ell_{\infty}$-norm of a vector $\bm{a}$, the Frobenius norm and the operator norm of a matrix $\bm{A}$, respectively. We denote $\mathbb{I}(\cdot)$ as the indicator function, $\|\bm{a}\|_0 := \sum_{i = 1}^d \mathbb{I}(\bm{a}_i \neq 0)$ as the $\ell_0$-norm (i.e., the total number of nonzero components), $\supp(\bm{a}) := \{i \in [d] ~|~ \bm{a}_i \neq 0 \}$ as the support set. We denote $\mathcal{V}_k := \{\bm{y} \in \mathbb{R}^d ~|~ \|\bm{y}\|_2 = 1, \|\bm{y}\|_0 \leq k\}$ as a set of $k$-sparse unit vectors of $d$-dimension, $\mathcal{N}(\mu, \sigma^2)$ as a Gaussian distribution with mean $\mu$ and covariance $\sigma^2$. For two sequences of non-negative reals $\{f_n\}_{n \geq 1}$ and $\{g_n\}_{n \geq 1}$, we use $f_n \lesssim g_n$ to indicate that there is a universal constant $C>0$ such that $f_n \leq C g_n$ for all $n \geq 1$. We further use standard order notation $f_n = O(g_n)$ to represent that $f_n \lesssim g_n$.

\subsection{Literature Review} \label{sec:literature} 

Sparse PCA and its adversarial perturbation have been studied extensively over the past decades. This subsection provides a brief literature review of papers that are most relevant to our contributions concerning MIPs and statistics.  

\textbf{MIPs for sparse PCA.} Briefly speaking, given an arbitrary centered sample matrix $\bm{X} \in \mathbb{R}^{n \times d}$, sparsity level parameter $k \leq d$, and using $\widehat{\bm{\Sigma}} := \frac{1}{n} \bm{X}^{\top} \bm{X}$ to denote uncorrupted sample covariance matrix, the vanilla sparse PCA problem $\max_{\bm{v} \in \mathcal{V}_k} \bm{v}^{\top} \widehat{\bm{\Sigma}} \bm{v}$ can be formulated into a mixed-integer quadratic program by introducing a set of binary variables $\{\bm{z}_i\}_{i = 1}^d$ such that $\mathcal{V}_k = \proj_{\bm{v}} \{(\bm{v}, \bm{z}) \in \mathbb{R}^d \times \{0,1\}^d ~|~ \|\bm{v}\|_2 =1, \|\bm{z}\|_1 \leq k, \bm{v}_i \in[ - \bm{z}_i, \bm{z}_i] ~ \forall ~ i \in [d] \}$. Unlike traditional PCA, it is well known that the above maximization problem for sparse PCA is NP-hard and inapproximable \citep{magdon2017np}. As a result, a series of existing works focus on improving the computational efficiency of this MIP. For example, \citet{berk2019certifiably} proposes a tailored branch-and-bound algorithm for the vanilla sparse PCA. \citet{gally2016computing,bertsimas2022solving,li2024exact} propose equivalent mixed-integer semi-definite programs (MISDPs) with distinct valid inequalities. In addition to finding exact solutions, researchers have actively investigated tractable convex relaxations. A common approach studies the semi-definite programming (SDP) relaxations for vanilla sparse PCA with provable bounds (e.g., \citet{d2004direct,amini2008high,d2008optimal,zhang2012sparse,d2014approximation,kim2022convexification}). More recently, \citet{chan2016approximability,dey2022using,dey2023solving,li2024exact} develop mixed-integer convex programs to approximate the vanilla sparse PCA and provided theoretical worst-case approximation guarantees. \emph{In contrast, this paper explores the properties of robust principal components with non-oblivious adversarial perturbations, where, to the best of our knowledge, existing MISDP relaxation cannot be directly applied.}

\textbf{Statistical results for sparse PCA.} The second category studies the sparse PCA in a purely statistical manner, i.e., without computational considerations. For instance, typical existing statistical results \citep{amini2008high,birnbaum2013minimax,cai2013sparse,vu2013minimax} aim to find
an estimator $\widehat{\bm{v}}$ with unit $\ell_2$-norm that recovers the ground truth under different metrics (e.g., vector recovery, support recovery, subspace recovery) with high probability. Beyond the vanilla sparse PCA, several results for sparse PCA with outliers or oblivious adversarial perturbations have been developed. For instance, \citet{awasthi2020estimating} proposes a computationally efficient algorithm that recovers the ground truth with theoretical guarantees under given corrupted samples. More recently, \citet{d2020sparse,d2021consistent,novikov2023sparse} propose efficient algorithms for solving Sparse PCA to achieve optimal vector recovery while being resilient against additive oblivious adversarial perturbations or under corrupted samples, and further designs new analysis techniques beyond covariance thresholding for sparse PCA with oblivious adversarial perturbations. \emph{As a comparison, this paper investigates theoretical performance guarantees of the robust principal component as non-oblivious adversarial perturbation increases. }

\section{Mixed-Integer Convex Programs}  \label{sec:SPCA-NOAP-refor} 

Given an adversarial uncertainty set $\mathcal{U}_{\|\cdot\|}(\rho)$, the sparse PCA-NOAP~\eqref{eq:SPCA-NOAP} can be written as the following max-min two-stage optimization problem, $\max_{\bm{v} \in \mathcal{V}_k}$ $\min_{\|\bm{E}\| \leq \rho} ~ \frac{1}{n} \|\bm{X} \bm{v} + \bm{E}\bm{v}\|_2^2.$ Considering its inner minimization problem, for any $\bm{v} \in \mathcal{V}_k$, we claim the following proposition holds. 

\begin{proposition} \label{prop:inner-min-refor}
For any $\bm{v} \in \mathcal{V}_k$, the inner minimization problem $\min_{\|\bm{E}\| \leq \rho}$ $\frac{1}{n} \|\bm{X} \bm{v} + \bm{E}\bm{v}\|_2^2$ can be written as $\frac{1}{n} \| \bm{X}\bm{v} - \proj_{\mathcal{F}(\bm{v}; \|\cdot\|, \rho)} (\bm{X} \bm{v}) \|_2^2$, where set $\mathcal{F}(\bm{v}; \|\cdot\|, \rho) := \left\{ \bm{E}\bm{v} \in \mathbb{R}^n ~|~ \|\bm{E}\| \leq \rho \right\}$ and $\proj_{\mathcal{F}}(\bm{x}) := \argmin_{\bm{u} \in \mathcal{F}} \| \bm{x} - \bm{u} \|_2^2$ projects a point $\bm{x}$ onto a given set $\mathcal{F}$. 
\end{proposition}
The proof of Proposition~\ref{prop:inner-min-refor} is straightforward and presented in Appendix~\ref{app:inner-min-refor}. Using Proposition~\ref{prop:inner-min-refor}, the two-stage sparse PCA-NOAP~\eqref{eq:SPCA-NOAP} can reduced to a one-stage maximization problem, 
\begin{align}
    \opt(\|\cdot\|, \rho) := \max_{\bm{v} \in \mathcal{V}_k} ~ \frac{1}{n} \| \bm{X}\bm{v} - \proj_{\mathcal{F}(\bm{v}; \|\cdot\|, \rho)} (\bm{X} \bm{v}) \|_2^2 ~~ .  \label{eq:SPCA-NOAP-ref}
\end{align}
Geometrically,  Formulation~\eqref{eq:SPCA-NOAP-ref} aims to find a $\bm{v}^{\tt RS}$ that maximizes the $\ell_2$-distance between $\bm{X}\bm{v}$ and corresponding projected point onto the uncertainty set $\mathcal{F}(\bm{v}; \|\cdot\|, \rho)$, where its optimal value is a function dependent on the uncertainty type measured by the norm $\|\cdot\|$ and uncertainty level $\rho \geq 0$. Notably, once the sample matrix satisfies that $\|\bm{X}\| \leq \rho$, the above projection outputs $\proj_{\mathcal{F}(\bm{v}; \|\cdot\|, \rho)} (\bm{X} \bm{v}) = \bm{X} \bm{v}$, which leads to the trivial case with $\opt(\|\cdot\|, \rho) = 0$. Therefore, in this paper, $\rho < \|\bm{X}\|$ is assumed to avoid this uninteresting case. In summary, Formulation~\eqref{eq:SPCA-NOAP-ref} gives a general/abstract one-stage reformulation of sparse PCA-NOAP~\eqref{eq:SPCA-NOAP}. Whether the above Formulation~\eqref{eq:SPCA-NOAP-ref} has a closed-form formulation only depends on whether the projection $\proj_{\mathcal{F}(\bm{v}; \|\cdot\|, \rho)} (\bm{X} \bm{v})$ has a closed-form solution. This paper focuses on two typical types of adversarial perturbations, e.g., \emph{sample-wise adversarial perturbation} and \emph{feature-wise adversarial perturbation}, to be detailed in Sections~\ref{sec:samplewise-MIP} and \ref{sec:featurewise-MIP}.   

\subsection{Uncertainty set $\mathcal{U}_{\|\cdot\|_{2 \rightarrow \infty}}(\rho)$ for sample-wise perturbation} \label{sec:samplewise-MIP}
With uncertainty set $\mathcal{U}_{\|\cdot\|_{2 \rightarrow \infty}}(\rho)$, the sample-wise adversarial perturbations are controlled by $\ell_{2 \rightarrow \infty}$-norm  ($\|\bm{E}\|_{2 \rightarrow \infty} := \max_{i = 1}^n \|\bm{E}_{i, :}\|_2 \leq \rho$), where the largest $\ell_2$-norm of all row vectors is upper bounded by $\rho$. In other words, a perturbation is placed on every sample $\bm{x}^i$ by adding an adversarial vector with $\ell_2$-norm at most $\rho$. The following Proposition \ref{prop:samplewise-set} gives a reformulation for such sample-wise perturbation.

\begin{proposition}\label{prop:samplewise-set} 
For any $\bm{v} \in \mathcal{V}_k$,  the resulting feasible set for adversarial perturbations $\mathcal{F}(\bm{v};\|\cdot\|_{2 \rightarrow \infty}, \rho)$ satisfies $\mathcal{F}(\bm{v};\|\cdot\|_{2 \rightarrow \infty}, \rho) = \left\{ \bm{u} \in \mathbb{R}^n ~|~ \|\bm{u}\|_{\infty} \leq \rho \|\bm{v}\|_2 \right\}$. 
\end{proposition}
Thus, projecting $\bm{X}\bm{v}$ onto $\mathcal{F}(\bm{v};\|\cdot\|_{2 \rightarrow \infty}, \rho)$ results in $\proj_{\mathcal{F}(\bm{v};\|\cdot\|_{2 \rightarrow \infty}, \rho)} (\bm{X} \bm{v}) = \bm{u} \in \mathbb{R}^n$ with a closed-form $\bm{u}_i = \left\{ 
    \begin{array}{lll}
        \langle \bm{x}^i, \bm{v} \rangle & \text{if } |\langle \bm{x}^i, \bm{v} \rangle| \leq \rho \|\bm{v}\|_2 \\
        \rho \|\bm{v}\|_2  & \text{if } \langle \bm{x}^i, \bm{v} \rangle > \rho \|\bm{v}\|_2 \\
        - \rho \|\bm{v}\|_2 & \text{if } \langle \bm{x}^i, \bm{v} \rangle < - \rho \|\bm{v}\|_2
    \end{array}
    \right.$ for all $i \in [n]$. 
Plugging the above result into the sparse PCA-NOAP formulation \eqref{eq:SPCA-NOAP-ref} gives
\begin{align*}
    \max_{\bm{v} \in \mathcal{V}_k} ~ \frac{1}{n} \sum_{i = 1}^n \ell_{\rho}(\langle \bm{x}^i, \bm{v}\rangle ) ~\text{with}~
    \ell_{\rho}(\langle \bm{x}, \bm{v}\rangle) := \left\{ 
    \begin{array}{lll}
        0 & \text{ if } |\langle \bm{x}, \bm{v}\rangle| \leq \rho \|\bm{v}\|_2 \\
        (\langle \bm{x}, \bm{v}\rangle - \rho\|\bm{v}\|_2 )^2 & \text{ if } \langle \bm{x}, \bm{v}\rangle > \rho \|\bm{v}\|_2 \\
        (\langle \bm{x}, \bm{v}\rangle + \rho\|\bm{v}\|_2)^2 & \text{ if } \langle \bm{x}, \bm{v}\rangle < - \rho \|\bm{v}\|_2
    \end{array}
    \right.
\end{align*}
Note that any $\bm{v} \in \mathcal{V}_k$ satisfies $\|\bm{v}\|_2 = 1$, thus the above optimization can be written as maximizing a convex function
\begin{align}
    \max_{\bm{v} \in \mathcal{V}_k} ~ \frac{1}{n} \sum_{i = 1}^n \ell_{\rho}^= (\langle \bm{x}^i, \bm{v}\rangle) ~\text{with}~\ell_{\rho}^{=}(\langle \bm{x}, \bm{v}\rangle) := \left\{
    \begin{array}{lll}
        0 & \text{ if } |\langle \bm{x}, \bm{v}\rangle| \leq \rho \\
        (\langle \bm{x}, \bm{v}\rangle - \rho )^2 & \text{ if } \langle \bm{x}, \bm{v}\rangle > \rho \\
        (\langle \bm{x}, \bm{v}\rangle + \rho )^2 & \text{ if } \langle \bm{x}, \bm{v}\rangle < - \rho
    \end{array}
    \right. \label{eq:samplewise-spca-1}
\end{align}
For convenience, in the rest of this paper, we use $\opt_k^{2 \rightarrow \infty}$ to denote the optimal value of Formulation~\eqref{eq:samplewise-spca-1}. Next, we propose a computationally tractable MIP reformulation for \eqref{eq:samplewise-spca-1}. Its main idea is to separate the objective function $\ell_{\rho}^{=}$ as a sum of a convex quadratic term and a concave term denoted by $\phi_{\rho}$.
\begin{align*}
    \ell_{\rho}^{=}(\langle \bm{x}^i, \bm{v}\rangle) = \langle \bm{x}^i, \bm{v}\rangle^2 + \phi_{\rho}(\langle \bm{x}^i, \bm{v}\rangle) ~~\text{with}~~ \phi_{\rho}(t) = \left\{
    \begin{array}{ll}
        - t^2 & \text{ if } ~  |t| \leq \rho \\
        - 2 \rho t + \rho^2 & \text{ if } ~ t > \rho \\
        2 \rho t + \rho^2 & \text{ if } ~ t < - \rho
    \end{array}
    \right. ~.
\end{align*}
Formulation~\eqref{eq:samplewise-spca-1} can be represented as $\max_{\bm{v} \in \mathcal{V}_k}  \frac{1}{n} \sum_{i = 1}^n \left[ \langle \bm{x}^i, \bm{v} \rangle^2 + \phi_{\rho}(\langle \bm{x}^i, \bm{v} \rangle) \right]$. Since the optimality of \eqref{eq:samplewise-spca-1} is always achieved with $\|\bm{v}\|_2 = 1$, the non-convex feasible set $\bm{v} \in \mathcal{V}_k$ can be relaxed to $\bm{v} \in \overline{\mathcal{V}}_k := \{ \bm{v} ~|~ \|\bm{v}\|_2 \leq 1, \|\bm{v}\|_0 \leq k \}$. Let $\widehat{\bm{\Sigma}} := \frac{1}{n} \bm{X}^{\top} \bm{X} = \frac{1}{n} \sum_{i = 1}^n \bm{x}^i (\bm{x}^i)^{\top}$ be the uncorrputed empirical covariance matrix. By introducing a new set of variables $\phi := \{\phi^i\}_{i = 1}^n$ as lower bounds for $\phi_{\rho} (\langle \bm{x}^i, \bm{v} \rangle)$, Formulation~\eqref{eq:samplewise-spca-1} becomes
\begin{align*}
    & ~ \begin{array}{rlll}
        \max\limits_{\bm{v}, \phi} &~ \bm{v}^{\top} \widehat{\bm{\Sigma}} \bm{v} + \frac{1}{n} \sum_{i = 1}^n \phi^i \\
        \text{s.t.} &~ \bm{v} \in \overline{\mathcal{V}}_k \\
        &~ \phi^i \leq \phi_{\rho} (\langle \bm{x}^i, \bm{v} \rangle) ~~ \forall ~ i \in [n] 
    \end{array} ~\Longleftrightarrow~ \begin{array}{rlll}
        \max\limits_{\bm{v}, \phi} &~ \obj + \frac{1}{n} \sum_{i = 1}^n \phi^i \\
        \text{s.t.} &~ \bm{v} \in \overline{\mathcal{V}}_k, ~ \obj = \bm{v}^{\top} \widehat{\bm{\Sigma}} \bm{v}\\
        &~ \phi^i \leq \phi_{\rho} (\langle \bm{x}^i, \bm{v} \rangle) ~~ \forall ~ i \in [n]
    \end{array} ~ .
\end{align*}
It remains to handle the non-convex constraint $\obj = \bm{v}^{\top} \widehat{\bm{\Sigma}} \bm{v}$. Let $\widehat{\bm{\Sigma}} := \sum_{i = 1}^d \lambda_j \bm{v}_j \bm{v}_j^{\top}$ be singular value decomposition of $\widehat{\bm{\Sigma}}$ with $\lambda_1 \geq \cdots \geq \lambda_d$. Using \emph{special-ordered set type-II (SOS-II)} \citep{wolsey2014integer}, for any subset $\mathcal{J} \subseteq [d]$, let the set for Piecewise Linear Upper Approximation (PLU) be 
\begin{align}
    \PLU(\mathcal{J}) := \left\{  (\bm{g}_j, \bm{\xi}_j, \eta_j)_{j \in \mathcal{J}} ~\left|~
    \begin{array}{lll}
        \bm{g}_j = \langle \bm{v}_j, \bm{v} \rangle = \sum_{\ell = - N}^N \theta_j^{\ell} \cdot \eta_j^{\ell} & ~ \forall ~ j \in \mathcal{J} \\
        \bm{\xi}_j = \sum_{\ell = - N}^N (\theta_j^{\ell})^2 \cdot \eta_j^{\ell} & ~ \forall ~ j \in \mathcal{J} \\
        \sum_{\ell = -N}^{N} \eta_j^{\ell} = 1 & ~ \forall ~ j \in \mathcal{J} \\
        \{\eta_j^{\ell} \}_{\ell = -N}^N \in \text{SOS-II} \cap \mathbb{R}_+^{2N + 1} & ~ \forall ~ j \in \mathcal{J} 
    \end{array} \right. 
    \right\} ~, \label{eq:PLU-set}
\end{align}
where $\{\theta_j^{\ell}\}_{\ell = -N}^N$ denotes a sequence of splitting points that evenly separate the interval $[-1,1]$ into $2N$ equal parts, i.e., $\theta_j^{\ell} = \frac{\ell}{N}$ for $\ell = -N, \ldots, N$. Based on the above definition, the following theorem introduces a computationally tractable MIP for \eqref{eq:samplewise-spca-1} with a provable upper bound. 
\begin{theorem}\label{thm:samplewise-spca-MIP}
Formulation~\eqref{eq:samplewise-spca-1} can be approximated by the following MIP.
\begin{align}
    \begin{array}{rlll}
        \ub_k^{2 \rightarrow \infty} := \max\limits_{\bm{v}, \phi, \bm{g}, \bm{\xi}, \eta} & \sum_{j = 1}^d \lambda_j \bm{\xi}_j  + \frac{1}{n} \sum_{i = 1}^n \phi^i \\
        \emph{s.t.}~ & \bm{v} \in \overline{\mathcal{V}}_k, (\bm{g}, \bm{\xi}, \eta) \in \PLU([d]) \\
        & \phi^i \leq \phi_{\rho} (\langle \bm{x}^i, \bm{v} \rangle) & ~ \forall ~ i \in [n] 
    \end{array} ~ , \label{eq:samplewise-CMIP}
\end{align}
where its optimal value $\ub_k^{2 \rightarrow \infty}$ satisfies $\opt_k^{2 \rightarrow \infty} \leq \ub_k^{2 \rightarrow \infty} \leq \opt_k^{2 \rightarrow \infty} + \frac{1}{4 N^2} \sum_{j = 1}^d \lambda_j $. 
\end{theorem}
The proof of Theorem~\ref{thm:samplewise-spca-MIP} is presented in Appendix~\ref{app:samplewise-spca-MIP}. One can improve the computational efficiency of the above formulation by reducing the number of SOS-II variables with an additional additive gap, as stated in Proposition~\ref{prop:samplewise-spca-MIP-r}. 
\begin{proposition}\label{prop:samplewise-spca-MIP-r}
Given a pre-determined integer $r \leq d$, one can further approximate the Formulation~\eqref{eq:samplewise-spca-1} by the following mixed
integer convex program.
\begin{align}
    \begin{array}{rlll}
        \ub_k^{2 \rightarrow \infty}(r) := \max\limits_{\bm{v}, \phi, \bm{g}, \bm{\xi}, \eta, \gamma} & \sum_{j = 1}^r \lambda_j \bm{\xi}_j  + \lambda_{r + 1} \gamma  + \frac{1}{n} \sum_{i = 1}^n \phi^i \\
        \emph{s.t.} ~~& \bm{v} \in \overline{\mathcal{V}}_k, (\bm{g}_j, \bm{\xi}_j, \eta_j)_{j \in [r]} \in \PLU([r])\\
        & \sum_{j = 1}^r \bm{g}_j^2 \leq 1 - \gamma, ~ \gamma \geq 0\\
        & \phi^i \leq \phi_{\rho} (\langle \bm{x}^i, \bm{v} \rangle) & ~ \forall ~ i \in [n] 
    \end{array} ~ . \label{eq:samplewise-MIP-r} 
\end{align}
The optimal value $\ub^{2 \rightarrow \infty}(r)$ satisfies $\opt^{2 \rightarrow \infty}
\leq \ub^{2 \rightarrow \infty}(r)\leq \opt^{2 \rightarrow \infty} + \frac{1}{4N^2} \sum_{j = 1}^r \lambda_j + \widehat{\gamma} (\lambda_{r + 1} - \lambda_d)$ with $\widehat{\gamma}$ being the value of $\gamma$ in an optimal solution of the above mixed integer convex program. 
\end{proposition}
The proof of Proposition~\ref{prop:samplewise-spca-MIP-r} is given in Appendix~\ref{app:samplewise-spca-MIP-r}.

\subsection{Uncertainty set $\mathcal{U}_{\|\cdot\|_{1 \rightarrow 2}}(\rho)$ for feature-wise perturbation} \label{sec:featurewise-MIP}
With uncertainty set $\mathcal{U}_{\|\cdot\|_{1 \rightarrow 2}}(\rho)$, the adversarial perturbations are controlled by $\ell_{1 \rightarrow 2}$-norm ($\|\bm{E}\|_{1 \rightarrow 2} := \max_{j = 1}^d \|\bm{E}_{:,j}\|_2 \leq \rho$), where the largest $\ell_2$-norm among all columns is upper bounded by $\rho$. In this case, the perturbation is placed on every feature/factor by adding an adversarial vector with $\ell_2$-norm at most $\rho$. Similarly, we can derive the following proposition. 

\begin{proposition}\label{prop:featurewise-set}
For any $\bm{v} \in \mathcal{V}_k$, the feasible set for adversarial perturbations $\mathcal{F}(\bm{v};\|\cdot\|_{1 \rightarrow 2}, \rho)$ satisfies $\mathcal{F}(\bm{v};\|\cdot\|_{1 \rightarrow 2}, \rho) = \left\{ \bm{u} \in \mathbb{R}^n ~|~ \|\bm{u}\|_2 \leq \rho \|\bm{v}\|_1 \right\}$. 
\end{proposition}
Therefore, projecting $\bm{X}\bm{v}$ onto $\mathcal{F}(\bm{v};\|\cdot\|_{1 \rightarrow  2}, \rho)$ gives $\proj_{\mathcal{F}(\bm{v};\|\cdot\|_{1 \rightarrow 2}, \rho)} (\bm{X} \bm{v}) = \bm{u}$ with $$\bm{u} = \left\{ 
    \begin{array}{lll}
        \bm{X} \bm{v} & \text{if } \|\bm{X} \bm{v}\|_2 \leq \rho \|\bm{v}\|_1 \\
        \rho \|\bm{v}\|_1 \bm{X}\bm{v} / \|\bm{X}\bm{v}\|_2 & \text{if } \|\bm{X} \bm{v}\|_2 > \rho \|\bm{v}\|_1
    \end{array}
    \right. ~ .$$ 
Plugging the above projection into the feature-wise perturbed sparse PCA-NOAP Formulation~\eqref{eq:SPCA-NOAP-ref} implies 
\begin{align}
    \opt^{1 \rightarrow 2}_k
    \overset{\tt (i)}{:=} &~ \max_{\bm{v} \in \mathcal{V}_k} ~ \frac{1}{n} \left( 
    \|\bm{X}\bm{v}\|_2 - \rho \|\bm{v}\|_1  \right)^2 ~~\text{s.t.}~~ \|\bm{X}\bm{v}\|_2 - \rho \|\bm{v}\|_1 \geq 0, \label{eq:emp-adv-spca-2}
\end{align} 
where the equality {\tt (i)} holds if and only if the feasible set is non-empty, i.e., $\{\bm{v} ~|~ \|\bm{Xv}\|_2 - \rho \|\bm{v}\|_1 \geq 0\} \cap \mathcal{V}_k \neq \emptyset$; otherwise, the adversarial perturbation is too large to capture any information from the sample covariance. Optimizing the feature-wise perturbed sparse PCA-NOAP~\eqref{eq:emp-adv-spca-2}, in this case, is not meaningful. Note that we can relax the constraint $\bm{v} \in \mathcal{V}_k$ to $\bm{v} \in \overline{\mathcal{V}}_k$ in Formulation~\eqref{eq:emp-adv-spca-2} by rescaling and still maintain the same optimal solution. Using SOS-II constraints, we can derive a computationally tractable MIP to approximate the feature-wise perturbed sparse PCA-NOAP~\eqref{eq:emp-adv-spca-2} with a provable guarantee.
\begin{theorem}\label{thm:featurewise-spca-MIP}
Formulation~\eqref{eq:emp-adv-spca-2} can be approximated by the following MIP. 
\begin{align}
    \begin{array}{rlll}
        \sqrt{\ub^{1 \rightarrow 2}_k} := \max\limits_{\bm{v}, t, \bm{y}, \bm{g}, \bm{\xi}, \eta} & \frac{1}{\sqrt{n}} t \\
        \emph{s.t.}~~ & \bm{v} \in \overline{\mathcal{V}}_k, (\bm{g}, \bm{\xi}, \eta) \in \PLU([d]) \\
        & n \sum_{j = 1}^d \lambda_j \bm{\xi}_j \geq (t + \rho \bm{y})^2 \\
        & t \geq 0, ~ \bm{y} \geq \|\bm{v}\|_1
    \end{array} ~.  \label{eq:featurewise-MIP} 
\end{align}
Moreover, $ \opt^{1 \rightarrow 2}_k\leq \ub^{1 \rightarrow 2}_k \leq \opt^{1 \rightarrow 2}_k + \frac{1}{4 N^2} \sum_{j = 1}^d \lambda_j  $.
\end{theorem}
The proof of Theorem~\ref{thm:featurewise-spca-MIP} can be found in Appendix~\ref{app:featurewise-spca-MIP}. Based on the techniques used in Proposition~\ref{prop:samplewise-spca-MIP-r}, one can improve computational efficiency by reducing the number of SOS-II variables, as detailed by  Proposition~\ref{prop:featurewise-spca-MIP-r} in Appendix~\ref{app:featurewise-spca-MIP-r}. 

\section{Statistical Results} \label{sec:stat-results}

This section presents the theoretical properties of optimal solutions from the proposed mixed-integer convex programs under the statistical sample-generating model in Assumption~\ref{assump:sample-model}.  
 
\begin{assumption}\label{assump:sample-model}
Every sample $\bm{x}^i$ is i.i.d. generated from an underlying Gaussian distribution $\mathcal{N}(\bm{0}, \bm{\Sigma})$. Here, the covariance matrix $\bm{\Sigma}$ is positive definite with eigengap $\lambda := \lambda^k_1 - \lambda^k_2 > 0$, where $\lambda^k_1, \lambda^k_2$ are the first and second largest $k$-sparse eigenvalue of $\bm{\Sigma}$, respectively.
\end{assumption}

\subsection{Statistical results for sample-wise perturbation} \label{sec:stat-samplewise}

Based on Assumption~\ref{assump:sample-model}, the population version of Formulation~\eqref{eq:samplewise-spca-1} becomes
\begin{align}
    \max_{\bm{v} \in \mathcal{V}_k} ~ & ~ \mathbb{E}_{\bm{x} \sim \mathcal{N}(\bm{0}, \bm{\Sigma})} [\ell_{\rho}(\bm{v}^{\top} \bm{x})] ~ , \label{eq:pop-adv-spca-1}
\end{align}
whose optimal solution is characterized by the following theorem.
\begin{theorem} \label{thm:pop-SPCA-NOAP-1}
Under Assumption~\ref{assump:sample-model}, Formulation~\eqref{eq:pop-adv-spca-1} and the population version of sparse PCA have the same set of optimal solutions i.e., \[\argmax_{\bm{v} \in \mathcal{V}_k} \mathbb{E}_{\bm{x} \sim \mathcal{N}(\bm{0}, \bm{\Sigma})} [\ell_{\rho}(\bm{v}^{\top} \bm{x})] ~=~ \argmax_{\bm{v} \in \mathcal{V}_k} \bm{v}^{\top} \bm{\Sigma} \bm{v}.\]
\end{theorem}
The proof of Theorem~\ref{thm:pop-SPCA-NOAP-1} is presented in Appendix~\ref{app:pop-SPCA-NOAP-1}. Theorem~\ref{thm:pop-SPCA-NOAP-1} implies that when samples are i.i.d. generated from Gaussian distributions, as the number of samples increases to infinity, the sample-wise perturbed sparse PCA reduces to the vanilla sparse PCA. In other words, sample-wise adversarial perturbation does not influence the sparse principal component in the population version. On the other hand, for the empirical version \eqref{eq:samplewise-spca-1}, as the adversarial perturbation parameter $\rho$ increases, the number of terms $\ell_{\rho}^=(\langle \bm{x}^i, \bm{v} \rangle)$ with non-zero value reduces exponentially with respect to $\rho$, implying a higher sample complexity to ensure vector recovery than that for vanilla sparse PCA. 

\begin{remark}
Theorem~\ref{thm:pop-SPCA-NOAP-1} can be generalized in the following way. Suppose all samples $\bm{x}_1, \ldots, \bm{x}_n$ are i.i.d. drawn from some distribution $\mathcal{D}$ with the properties: (i) $\mathcal{D}$ has a mean of $\bm{0}_d$ and covariance of $\bm{\Sigma} \succeq  \bm{0}_{d \times d}$; (ii) given a unit vector $\bm{v}$, the density function $f_{\mathcal{D}}$ of random variable $x_v := \bm{v}^{\top} \bm{x}_i$ is a univariate function that only depends on its standard derivation $\sigma_v := \sqrt{\bm{v}^{\top}\bm{\Sigma} \bm{v}}$, that is to say, its density function can be written as $f_{\mathcal{D}}(\cdot; \sigma_v)$; and (iii) the truncated variance function $h_{\rho}(\sigma_v) := \int_{x_v \geq \rho} (x_v^2 - \rho)^2 f_{\mathcal{D}}(x_v; \sigma_v) \mathrm{d} x_v$ is monotonically non-decreasing, i.e., $h_{\rho}(\sigma_v') \geq h_{\rho}(\sigma_v)$ if $\sigma_v' \geq \sigma_v$, then optimizing the population version of robust sparse PCA is equivalent to optimizing the classical sparse PCA problem for any fixed positive $\rho$, i.e., $\argmax_{\bm{v} \in \mathcal{V}_k} ~ h_{\rho} (\sqrt{\bm{v}^{\top} \bm{\Sigma} \bm{v}}) = \argmax_{\bm{v} \in \mathcal{V}_k} \sqrt{\bm{v}^{\top} \bm{\Sigma} \bm{v}} =  \argmax_{\bm{v} \in \mathcal{V}_k} \bm{v}^{\top} \bm{\Sigma} \bm{v}$. 
\end{remark}

\begin{remark}
We would like to point out that such an equivalent condition does not hold when the sample set is arbitrarily generated. For example, given samples  $\bm{x}_1 = (1 ~|~ 0)^{\top}$ and $\bm{x}_2 = (1/2 ~|~ \sqrt{3}/2)^{\top}$, the corresponding empirical covariance matrix is 
$$\bm{\Sigma} := \begin{pmatrix}
        5/8 & \sqrt{3}/8 \\
        \sqrt{3}/8 & 3/8 
    \end{pmatrix} = \frac{1}{2} \left[
    \begin{pmatrix}
        1 \\
        0 \\
    \end{pmatrix} \begin{pmatrix}
        1 & 0\\
    \end{pmatrix} + 
    \begin{pmatrix}
        1/2 \\
        \sqrt{3}/2\\
    \end{pmatrix} \begin{pmatrix}
        1/2 & \sqrt{3}/2 \\
    \end{pmatrix}
    \right].$$
Suppose $d = k = 2$ and $\rho = \frac{\sqrt{3}}{2} + \epsilon$ for a small $\epsilon > 0$, we have $ \argmax_{\bm{v} \in \mathcal{V}_k} \bm{v}^{\top} \bm{\Sigma} \bm{v}=(\sqrt{3}/2, 1/2)^{\top}$ while $\argmax_{\bm{v} \in \mathcal{V}_k} ~ \text{\eqref{eq:samplewise-spca-1}} = \{(1,0)^{\top}, (1/2, \sqrt{3}/2)^{\top}\} $. They have distinct optimal solutions and optimal values.  
\end{remark}
As a direct corollary from Theorem~\ref{thm:pop-SPCA-NOAP-1}, finding the optimal solution of Formulation~\eqref{eq:pop-adv-spca-1} can be reduced to solving the vanilla sparse PCA problem $\max_{\bm{v} \in \mathcal{V}_k} ~ \bm{v}^{\top} \bm{\Sigma} \bm{v}$. Therefore, any algorithms for the vanilla sparse PCA (i.e., truncated power method \citep{yuan2013truncated}, diagonal thresholding method \citep{amini2008high}, covariance thresholding method \citep{deshp2016sparse}, SDP-based methods \citep{amini2008high}) could be used for solving Formulation~\eqref{eq:pop-adv-spca-1}.

\subsection{Statistical results for feature-wise perturbation} \label{sec:stat-featurewise} 

When the perturbation is feature-wise, under Assumption~\ref{assump:sample-model}, we show that the population version of Formulation~\eqref{eq:emp-adv-spca-2} satisfies the following property. 

\begin{proposition} \label{prop:featurewise-spca-pop}
Under Assumption~\ref{assump:sample-model}, given any $\rho > 0$, the population version of Formulation~\eqref{eq:emp-adv-spca-2} is $\max_{\bm{v} \in \mathcal{F}(\rho / \sqrt{n} )} ~ ( \sqrt{\bm{v}^{\top} \bm{\Sigma} \bm{v} } - \frac{\rho}{\sqrt{n}} \|\bm{v}\|_1 )^2 - 2  \frac{\rho}{\sqrt{n}} \|\bm{v}\|_1 \sqrt{\bm{v}^{\top} \bm{\Sigma} \bm{v}} \cdot O ( 1/n )$ with feasible set $\mathcal{F}(\rho / \sqrt{n} ) := \{ \bm{v} \in \mathcal{V}_s^d ~\left|~ \sqrt{\bm{v}^{\top} \bm{\Sigma} \bm{v}} \cdot \left( 1 + O\left( \frac{1}{n} \right) \right) \geq \frac{\rho}{\sqrt{n}} \|\bm{v}\|_1 \right. \}$. 
\end{proposition}
The proof of Proposition~\ref{prop:featurewise-spca-pop} is given in Appendix~\ref{app:featurewise-spca-pop}. As hinted from existing statistical results \citep{novikov2023sparse}, setting $\rho = O(\sqrt{n / k})$ leads to, on average, an almost constant upper bound on the adversarial perturbation for every component of each sample. Under such a setting, the last term $2 \frac{\rho}{\sqrt{n}} \|\bm{v}\|_1 \sqrt{\bm{v}^{\top} \bm{\Sigma} \bm{v}} \cdot O \left( 1/n \right) \approx O(1/n)$ in objective is significantly smaller than the first quadratic term, as the number of samples $n$ increases. By ignoring the last term, we consider the following problem~\eqref{eq:pop-featurewise-spca} as a tight upper approximation for the original population version of Formulation~\eqref{eq:emp-adv-spca-2}, 
\begin{align}
    \widehat{\bm{v}} := & ~ \argmax_{\bm{v} \in \mathcal{V}_k} ~ ( \sqrt{\bm{v}^{\top} \bm{\Sigma} \bm{v}} -  \frac{\rho}{\sqrt{n}} \|\bm{v}\|_1 )^2 ~~\text{s.t.}~~ \sqrt{\bm{v}^{\top} \bm{\Sigma} \bm{v}} \geq \frac{\rho}{\sqrt{n}} \|\bm{v}\|_1, \label{eq:pop-featurewise-spca}
\end{align}
which is further equivalent to $\argmax_{\bm{v} \in \mathcal{V}_k} \sqrt{\bm{v}^{\top} \bm{\Sigma} \bm{v}} -  \frac{\rho}{\sqrt{n}} \|\bm{v}\|_1$, where such a equivalence holds when $\left\{ \bm{v} \in \mathcal{V}_k ~|~ \sqrt{\bm{v}^{\top} \bm{\Sigma} \bm{v}} \geq \frac{\rho}{\sqrt{n}} \|\bm{v}\|_1 \right\} \neq \emptyset$. 

\paragraph{Comparison with existing statistical results.} We begin with an additional assumption needed to establish further statistical results.  

\begin{assumption}\label{assump:spiked-Wishart-model}
Suppose Assumption~\ref{assump:sample-model} holds, and its covariance matrix $\bm{\Sigma}$ is of the following spiked format, i.e.,  $\bm{\Sigma} = \lambda \bm{v}_* \bm{v}_*^{\top} + \bm{I}_d$, where the eigengap $\lambda > 0$ and ground truth $\bm{v}_*$ is a unit $\ell_2$-norm vector with support set $S$ of size $k$.
\end{assumption}
Based on Assumption~\ref{assump:spiked-Wishart-model}, an existing result (Theorem 4.1 in \citet{novikov2023sparse}, see Appendix~\ref{app:vector-recovery} for a formal statement) ensures that: there exist algorithms with output a unit vector $\widehat{\bm{v}}$ that recovers the ground truth $\bm{v}_*$, i.e., $|\langle \widehat{\bm{v}}, \bm{v}_*\rangle| \geq 1 - \delta$, with high probability, while samples are corrupted under relatively small (``oblivious'') adversarial perturbations $\rho \leq O(\delta^6) \cdot \min\{\sqrt{\lambda}, \lambda\} \sqrt{n / k}$ and some specific parameter regime.

\begin{theorem}\label{thm:vector-recovery}
Suppose Assumption~\ref{assump:spiked-Wishart-model} holds, for any $\rho \leq c \cdot \delta \lambda \sqrt{\frac{n}{k}}$ with $c$ being a positive constant, the optimal solution $\widehat{\bm{v}} := \argmax_{\bm{v} \in \mathcal{V}_k} ~ \sqrt{\bm{v}^{\top} \bm{\Sigma} \bm{v} } - \frac{\rho}{\sqrt{n}} \|\bm{v}\|_1$ to Formulation~\eqref{eq:pop-featurewise-spca} ensures $|\langle \widehat{\bm{v}}, \bm{v}_* \rangle | \geq 1 - \delta$, where $\bm{v}_*$ denotes the optimal solution of the sparse PCA problem, i.e., $\bm{v}_* := \argmax_{\bm{v} \in \mathcal{V}_k} \bm{v}^{\top} \bm{\Sigma} \bm{v}$. 
\end{theorem}
The proof of Theorem~\ref{thm:vector-recovery} is presented in Appendix~\ref{app:vector-recovery}. Based on a weaker condition than the existing result (Theorem 4.1 in \citet{novikov2023sparse}), Theorem~\ref{thm:vector-recovery} shows that the optimal solution obtained from Formulation~\eqref{eq:pop-featurewise-spca} recovers the ground truth under a larger upper bound of adversarial perturbation parameter. 

\paragraph{Strong and weak signal under non-oblivious perturbation.} This part studies the properties of the robust sparse principal component under feature-wise adversarial perturbation as the parameter $\rho$ increases.   

\begin{assumption} \label{assump:strong-weak}
Based on Assumption~\ref{assump:spiked-Wishart-model}, suppose the support set $S_*$ for ground truth $\bm{v}_*$ can be partitioned into two non-overlap parts: a relatively strong signal part $S_1$ of size $k_1$ and a relatively weak signal part $S_2$ of size $k_2$ such that $[\bm{v}_*]_i = \sqrt{c / k_1}$ for $i \in S_1$ and $[\bm{v}_*]_{i} = \sqrt{(1 - v)/(k - k_1)}$ for $i \in S_2$ with $k_1 \ll k - k_1 \leq k$ and $c$ is some pre-determined constant in $(1/2,1)$. 
\end{assumption} 
For a given $\rho$, recall $\widehat{\bm{v}}$ the optimal solution of Formulation~\eqref{eq:pop-featurewise-spca}. We are poised to propose the main theorem of this part. 

\begin{theorem}\label{thm:strong-weak-signal}
Under Assumption~\ref{assump:strong-weak}, the behavior of $\widehat{\bm{v}}$ satisfies the following:

\noindent \textbf{1.}  \textbf{Recovery stage.} When $\rho \in [0, ~ O(\min\{\sqrt{\lambda}, \lambda\} \sqrt{ n / k_2})]$, $\widehat{\bm{v}}$ recovers the ground truth $\bm{v}_*$ defined in Theorem~\ref{thm:vector-recovery}. 

\noindent \textbf{2.} \textbf{Robust stage.} When $\rho \in (O(\min\{\sqrt{\lambda}, \lambda\} \sqrt{n / k_2}), ~ O(\min\{\sqrt{\lambda}, \lambda\} \sqrt{ n / k_1})]$, $\widehat{\bm{v}}$ recovers the strong signal part in $\bm{v}_*$ while eliminating the weak signal. 

\noindent \textbf{3.} \textbf{Overly perturbed stage.} When $\rho \in (O(\min\{\sqrt{\lambda}, \lambda\} \sqrt{ n / k_1}), ~ + \infty)$, adversarial perturbation is too large to recover enough information of the ground truth. 
\end{theorem}
Theorem~\ref{thm:strong-weak-signal} provides a quantitative description on how optimal solution $\widehat{\bm{v}}$ would change as the adversarial perturbation parameter increases. For people of independent interest, the proof of Theorem~\ref{thm:strong-weak-signal}, provided in Appendix~\ref{app:strong-weak-signal}, can generally be divided into the following three main steps. Step-1 provides a simplified reformulation of the objective function for Formulation~\eqref{eq:pop-featurewise-spca} under Assumption~\ref{assump:strong-weak}. Step-2 then demonstrates that such optimality of this formulation can be achieved at only four distinct points. Consequently, determining optimality reduces to comparing objective values among these four points. Based on this comparison, we derive conditions (lower and upper bounds) for $\rho$, ensuring that one of these points is guaranteed to be optimal. Finally, building upon Step-2, Step-3 establishes the connections between adversarial perturbation parameter $\rho$ and the behavior of $\widehat{\bm{v}}$ as presented in Theorem~\ref{thm:strong-weak-signal}.

\subsection{Numerical simulations} \label{sec:numerical}

In this section, we conduct numerical simulations on synthetic samples (generated from spiked Wishart model) to validate the theoretical findings presented in Section~\ref{sec:stat-featurewise} and demonstrate the tightness of proposed mixed-integer program formulations compared with dual baselines. Due to the space limit, we put the detailed descriptions of the implemented MIP-based approaches (i.e., \MIP ~and \MIPr) and baselines (i.e., \spca ~as dual baseline, \PPM ~as primal heuristic baseline), and information of hardware and software in the Appendix~\ref{app:baselines}, and additional numerical results in Appendix~\ref{app:additional-numerical}. 

\textbf{Sample matrix:} The synthetic samples are generated as follows. Every $d$-dimensional sample $\bm{x}^i$ for $i \in [n]$ is i.i.d. drawn from the spiked Wishart model with general $k$-sparse ground truth $\bm{v}_*^{\tt ks}$ or $k$-sparse ground truth with strong and weak signal $\bm{v}_*^{\tt sw}$ as presented in Assumptions~\ref{assump:spiked-Wishart-model} and \ref{assump:strong-weak}, respectively. We use $\bm{X} = (\bm{x}^1 ~|~ \cdots ~|~ \bm{x}^n)^{\top} \in \mathbb{R}^{n \times d}$ to denote the uncorrupted sample matrix.

\textbf{Performance measures:} We measure numerical performance by the following three metrics. \textbf{1.} \emph{Primal-dual gap}, defined as $\gap := (\ub - \lb) / \lb$, is used to measure the relative tightness of upper (dual) bounds for any given methods. In our numerical simulations, the upper bound is achieved from three upper (dual) bound methods, i.e., $\{\ub^{\MIP}, \ub^{\MIPr}, \ub^{\spca}\}$, and the lower bound $\lb$ denotes the best primal (lower) bound from the primal heuristics and feasible MIP solutions. \textbf{2.} \emph{Angle (cosine value) between true and computed solution}, defined as $\ang := |\langle \widehat{\bm{v}}, \bm{v}_* \rangle| / (\|\widehat{\bm{v}}\|_2 \|\bm{v}_*\|_2)$, is used to measure the similarity (in direction) between the given vector $\widehat{\bm{v}}$ and the ground truth, where $\widehat{\bm{v}} \in \{\bm{v}^{\MIP / \MIPr}, \bm{v}^{\spca}, \bm{v}^{\PPM}\}$ is a solution obtained from corresponding methods. Moreover, in strong and weak signal part, we use $\ang^{\strong} := |\langle [\widehat{\bm{v}}]_{S_1}, [\bm{v}_*]_{S_1} \rangle| / (\|[\widehat{\bm{v}}]_{S_1}\|_2 \|[\bm{v}_*]_{S_1}\|_2)$ and $\ang^{\weak} := |\langle [\widehat{\bm{v}}]_{S_2}, [\bm{v}_*]_{S_2} \rangle| / (\|[\widehat{\bm{v}}]_{S_2}\|_2 \|[\bm{v}_*]_{S_2}\|_2)$ to denote the cosine values for the strong and weak signal part between the ground truth and computed solutions, respectively. \textbf{3.} \emph{Support recovery rate}, defined as $\rate^{\supp} := |\widehat{S} \cap S_*| / k$, is used to measure the support recovery rate, where we use $ \widehat{S} := \supp(\widehat{\bm{v}})$ to denote the support of any given vector $\widehat{\bm{v}}$ and $S_*$ to denote the true support. Again, for the strong and weak signal parts, we use $\rate^{\strong} := |\widehat{S}_1 \cap S_1|/k_1$ and $\rate^{\weak} := |\widehat{S}_2 \cap S_2|/k_2$ to denote the relative support recovery rate to the strong and weak signal parts of the truth support set, respectively.

\textbf{Parameter setting:} In the simulation, we set sample dimension $d =100$, sparsity level $k=5$, eigenvalue gap $\lambda =3$. We generate $n \in \{100, 500\}$ i.i.d. samples from spiked Wishart model as described in Assumption~\ref{assump:spiked-Wishart-model} or Assumption~\ref{assump:strong-weak}. To cancel the influences from the number of samples $n$ and sparsity level $k$, the rescaled/normalized parameter $\bar{\rho}:= \rho \sqrt{k / n}$ for adversarial perturbation is set to be $\bar{\rho} \in \{0, 0.5, \ldots, 4.5\}$. For the implemented MIP method, we reduce the dimensionality of the original covariance matrix from $d$ to $\bar{d} \in \{15\}$. We set the number of splitting points $N \in \{3\}$, and the Gurobi time-limit equals to $1800\text{sec}$. Please refer to Appendix~\ref{app:baselines} for more parameter settings and implementation details of baselines.

\begin{figure}[!h]
\centering
\begin{subfigure}
\centering
\includegraphics[width=0.23\linewidth]{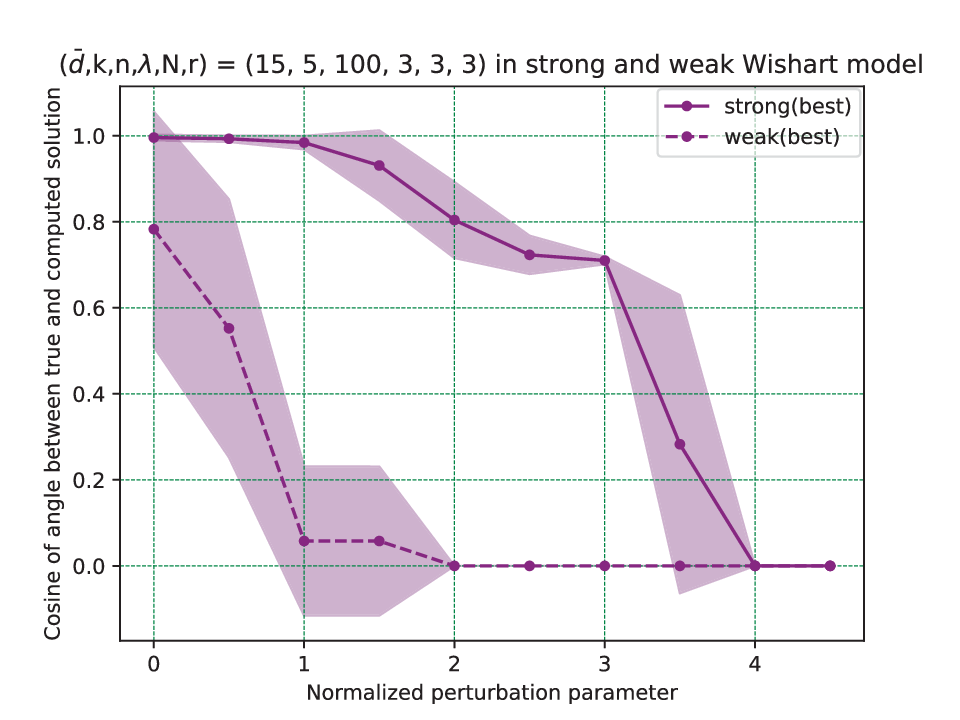}
\end{subfigure}
\hfill
\begin{subfigure}
\centering
\includegraphics[width=0.23\linewidth]
{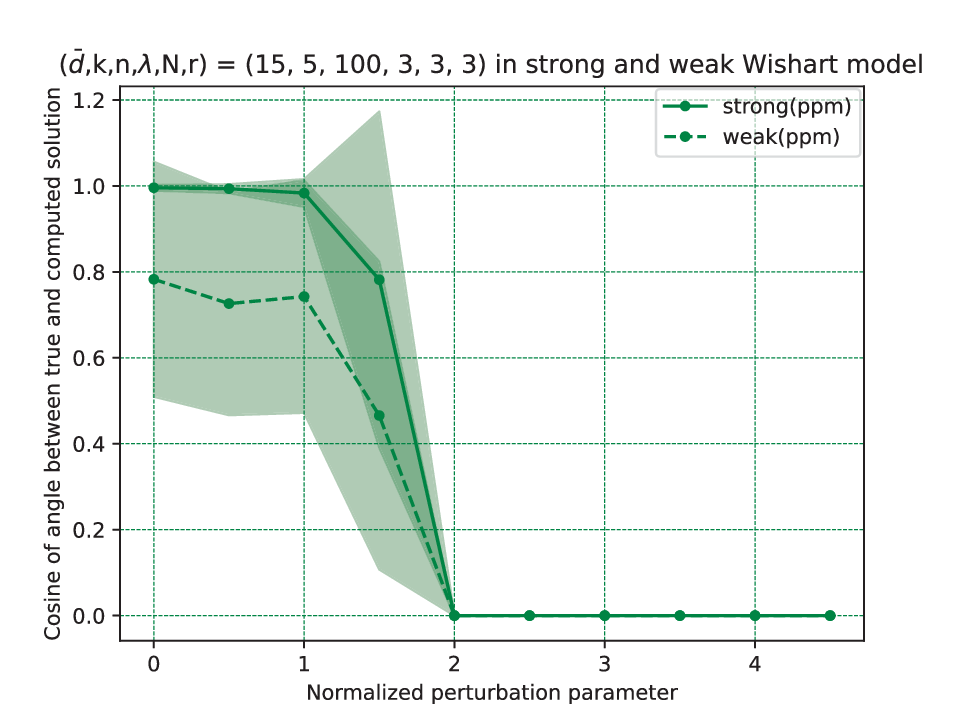}
\end{subfigure}
\hfill
\begin{subfigure}
\centering
\includegraphics[width=0.23\linewidth]
{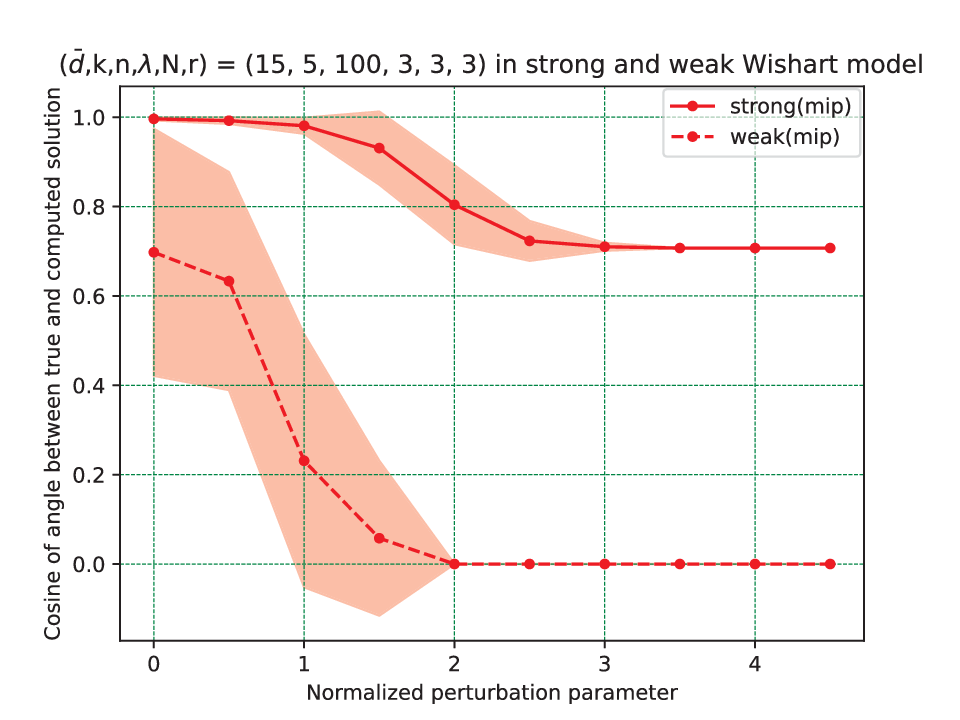}
\end{subfigure}
\hfill
\begin{subfigure}
\centering
\includegraphics[width=0.23\linewidth]
{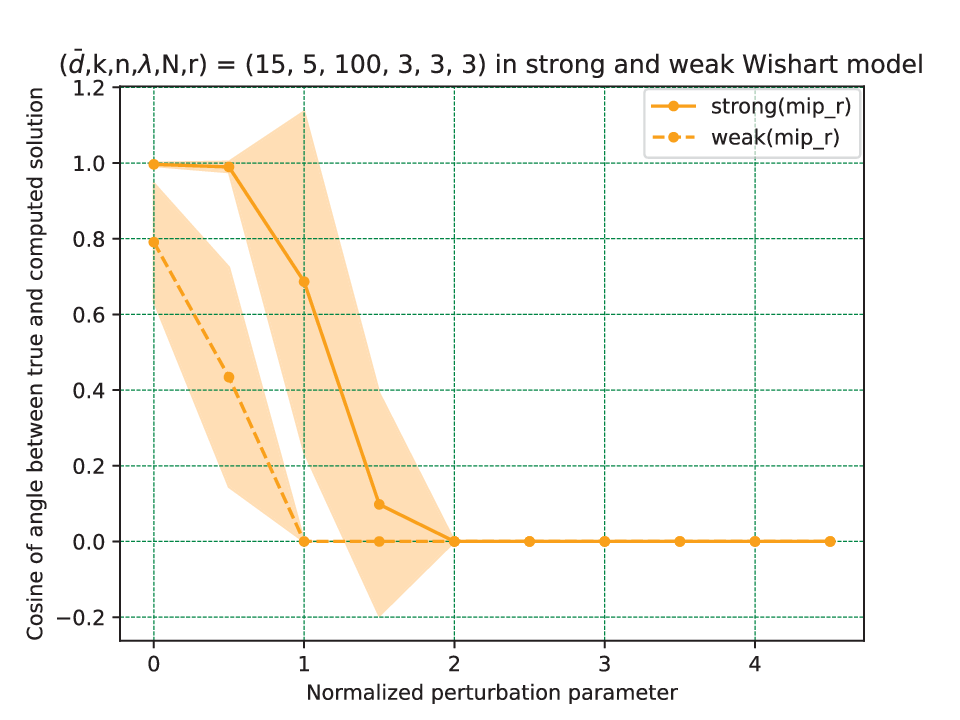}
\end{subfigure}

\centering
\begin{subfigure}
\centering
\includegraphics[width=0.23\linewidth]{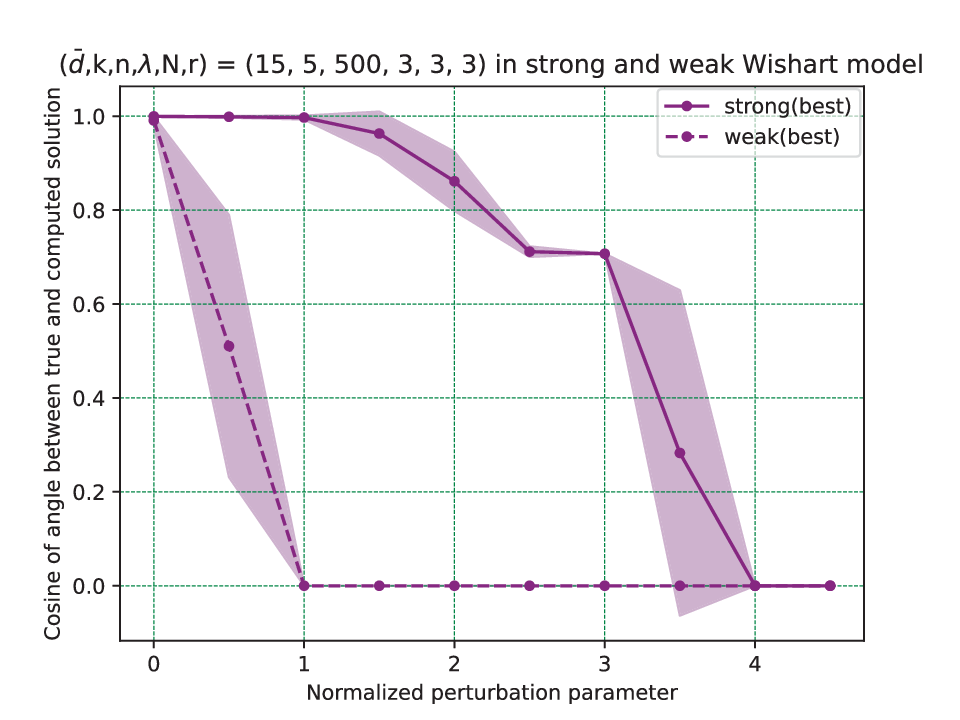}
\end{subfigure}
\hfill
\begin{subfigure}
\centering
\includegraphics[width=0.23\linewidth]
{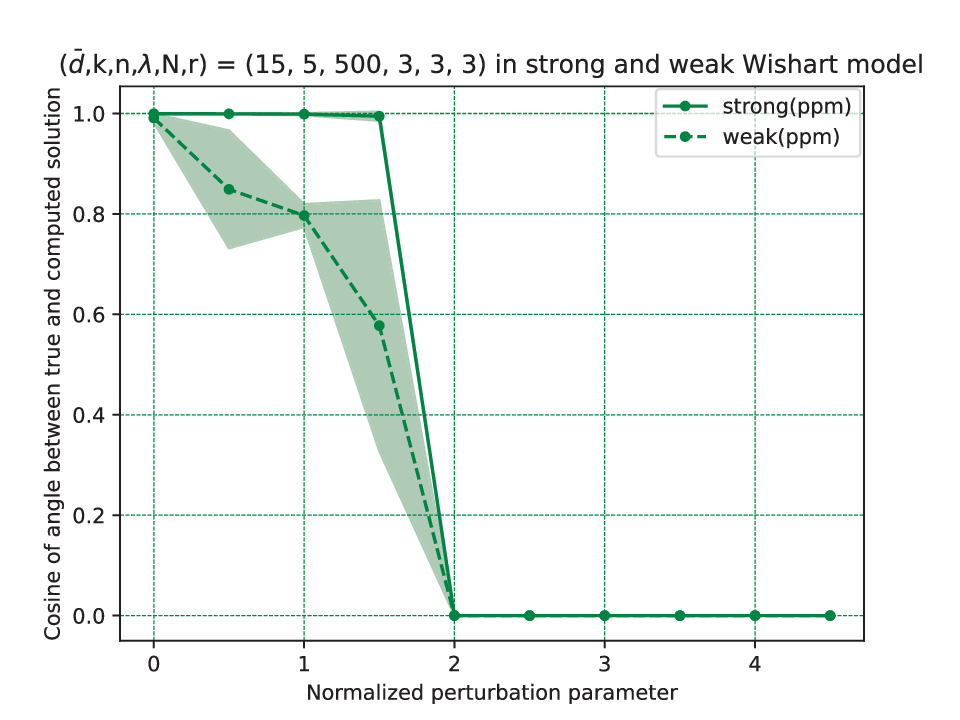}
\end{subfigure}
\hfill
\begin{subfigure}
\centering
\includegraphics[width=0.23\linewidth]
{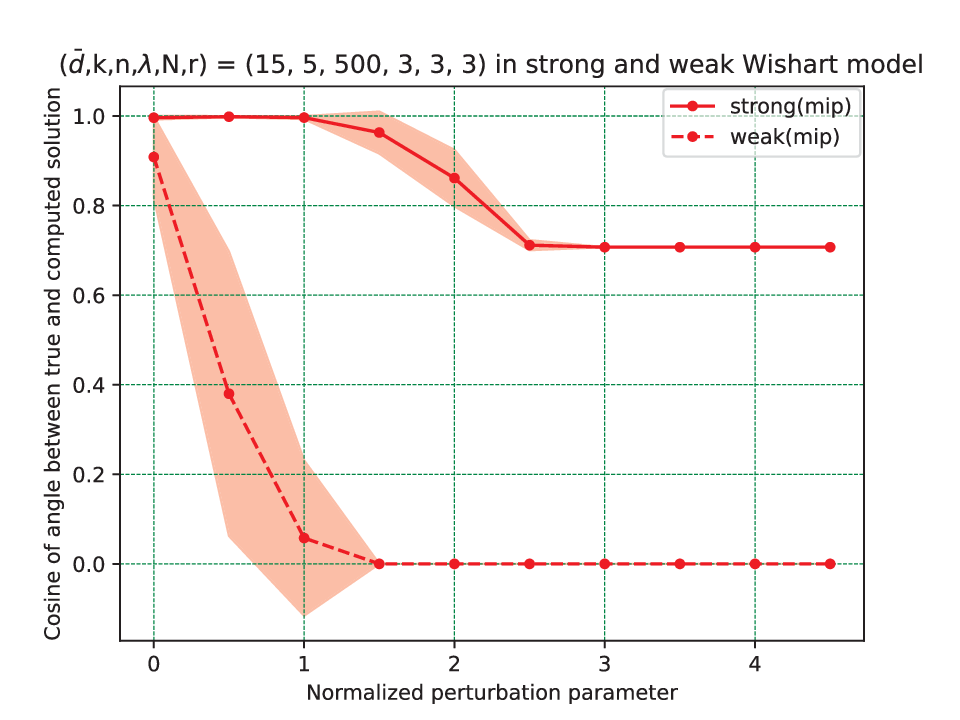}
\end{subfigure}
\hfill
\begin{subfigure}
\centering
\includegraphics[width=0.23\linewidth]
{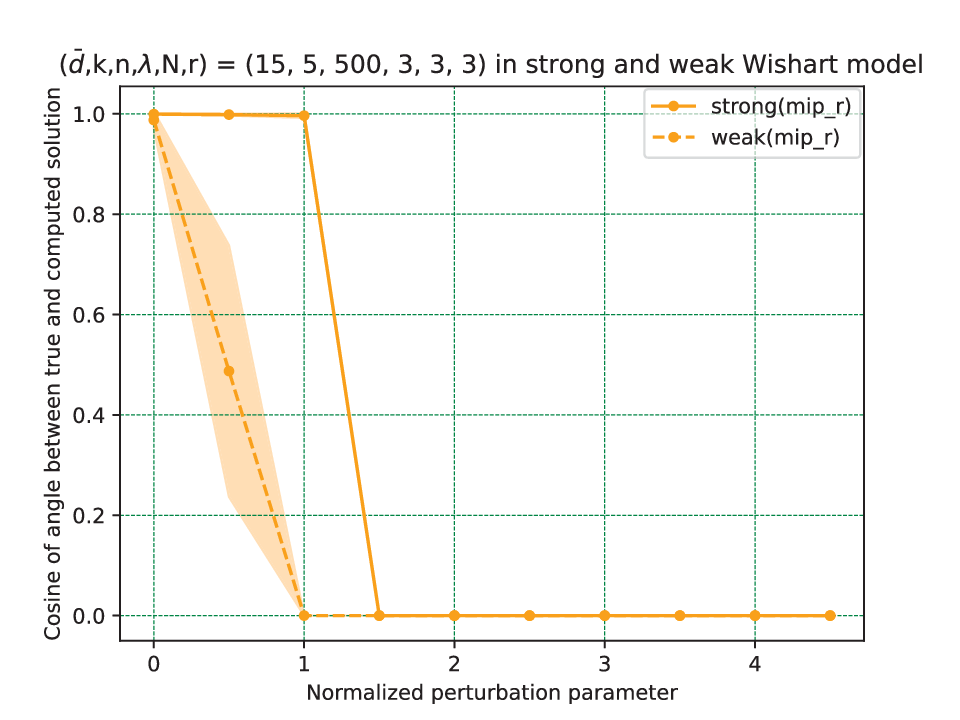}
\end{subfigure}
\caption{\textbf{Numerical results to validate Theorem~\ref{thm:strong-weak-signal}.} We plot the cosine values of angles between true and computed solution for strong and weak signal part v.s. the normalized perturbation parameter $\bar{\rho}$ under a particular parameter setting $(\bar{d},k,n,\lambda,N,r)$. The first and second row reports the numerical results with $n = 100$ and $n = 500$, respectively. For a specific normalized perturbation parameter $\bar{\rho}$, we generate 10 independent sample matrices based on the above parameter setting. Over these 10 independent trials, the solid curve in each panel corresponds to the averaged cosine values of the strong signal part; similarly, the dashed curve denotes the averaged cosine values of the weak signal part; the shaded parts represent the empirical standard deviations. The second/third/fourth column denotes the cosine values with respect to corresponding methods. The first column, denoted by ``{\tt best}'', represents the cosine values for the solutions with the best objective value of \eqref{eq:pop-featurewise-spca} among all three methods.
We can observe that: when $\bar{\rho}$ is relatively small, both strong and weak parts can be recovered with large cosine values, which follows the recovery stage described in Theorem~\ref{thm:strong-weak-signal}; while $\bar{\rho}$ increases, the weak signal part dramatically drops to zero and the strong signal part stays relatively higher cosine value as described in the robust stage in Theorem~\ref{thm:strong-weak-signal}; finally, if $\bar{\rho}$ grows beyond some threshold, adversarial perturbation are too large to capture any information of the ground truth. 
}
\label{figure:syn-data-2} 
\end{figure}

\textbf{Numerical results:} We compare the results under these three metrics under the following two parameter settings, i.e., $(\bar{d},k,n,\lambda,N,r) = (15,5,500,3,3,3)$ and $(\bar{d},k,n,\lambda,N,r) = (15,5,100,3,3,3)$. In general, Figure~\ref{figure:syn-data-2} validates the theoretical findings presented in Theorem~\ref{thm:strong-weak-signal}. Due to the space limit, we move Figure~\ref{figure:syn-data-1} to Appendix~\ref{app:additional-numerical}, which demonstrates the relative accuracy of computing tight upper bounds via our proposed MIP methods, \MIP \ and \MIPr. The code and data will be made publicly available on GitHub later. 

\section{Summary and Future Directions} \label{sec:summary}
This paper investigates the problem of sparse PCA-NOAP and proposes computationally efficient MIP approximations with theoretical worst-case guarantees for two critical variants of sparse PCA-NOAP. We then establish theoretical properties of the optimal solutions under non-oblivious adversarial perturbations when the samples are i.i.d. generated from a spiked Wishart model. Furthermore, numerical simulations are provided to validate our theoretical results and demonstrate the advantages of the proposed reformulations.

To conclude, we highlight several promising directions for future research. First, developing and analyzing polynomial-time algorithms for computing feasible solutions to sparse PCA-NOAP with guarantees is an essential next step. Second, exploring continuous convex relaxations for sparse PCA-NOAP with provable approximation ratios presents another promising avenue. Third, extending our results to generalized eigenvalue problems with adversarial perturbations would broaden the applicability of our approach. We anticipate that our work on sparse PCA-NOAP will inspire further advancements in PCA-related research, with practical and theoretical benefits for the field.


%
%
%
\bibliographystyle{abbrvnat}
\bibliography{reference}

\begin{thebibliography}{28}
\providecommand{\natexlab}[1]{#1}
\providecommand{\url}[1]{\texttt{#1}}
\expandafter\ifx\csname urlstyle\endcsname\relax
  \providecommand{\doi}[1]{doi: #1}\else
  \providecommand{\doi}{doi: \begingroup \urlstyle{rm}\Url}\fi

\bibitem[Amini and Wainwright(2008)]{amini2008high}
A.~A. Amini and M.~J. Wainwright.
\newblock High-dimensional analysis of semidefinite relaxations for sparse principal components.
\newblock In \emph{2008 IEEE international symposium on information theory}, pages 2454--2458. IEEE, 2008.

\bibitem[Awasthi et~al.(2020)Awasthi, Chen, and Vijayaraghavan]{awasthi2020estimating}
P.~Awasthi, X.~Chen, and A.~Vijayaraghavan.
\newblock Estimating principal components under adversarial perturbations.
\newblock In \emph{Conference on Learning Theory}, pages 323--362. PMLR, 2020.

\bibitem[Bai et~al.(2021)Bai, Luo, Zhao, Wen, and Wang]{bai2021recent}
T.~Bai, J.~Luo, J.~Zhao, B.~Wen, and Q.~Wang.
\newblock Recent advances in adversarial training for adversarial robustness.
\newblock \emph{arXiv preprint arXiv:2102.01356}, 2021.

\bibitem[Berk and Bertsimas(2019)]{berk2019certifiably}
L.~Berk and D.~Bertsimas.
\newblock Certifiably optimal sparse principal component analysis.
\newblock \emph{Mathematical Programming Computation}, 11:\penalty0 381--420, 2019.

\bibitem[Bertsimas et~al.(2022)Bertsimas, Cory-Wright, and Pauphilet]{bertsimas2022solving}
D.~Bertsimas, R.~Cory-Wright, and J.~Pauphilet.
\newblock Solving large-scale sparse pca to certifiable (near) optimality.
\newblock \emph{Journal of Machine Learning Research}, 23\penalty0 (13):\penalty0 1--35, 2022.

\bibitem[Birnbaum et~al.(2013)Birnbaum, Johnstone, Nadler, and Paul]{birnbaum2013minimax}
A.~Birnbaum, I.~M. Johnstone, B.~Nadler, and D.~Paul.
\newblock Minimax bounds for sparse pca with noisy high-dimensional data.
\newblock \emph{Annals of statistics}, 41\penalty0 (3):\penalty0 1055, 2013.

\bibitem[CAI et~al.(2013)CAI, MA, and WU]{cai2013sparse}
T.~T. CAI, Z.~MA, and Y.~WU.
\newblock Sparse pca: Optimal rates and adaptive estimation.
\newblock \emph{The Annals of Statistics}, 41\penalty0 (6):\penalty0 3074--3110, 2013.

\bibitem[Chan et~al.(2016)Chan, Papailliopoulos, and Rubinstein]{chan2016approximability}
S.~O. Chan, D.~Papailliopoulos, and A.~Rubinstein.
\newblock On the approximability of sparse pca.
\newblock In \emph{Conference on Learning Theory}, pages 623--646. PMLR, 2016.

\bibitem[d'Aspremont et~al.(2004)d'Aspremont, Ghaoui, Jordan, and Lanckriet]{d2004direct}
A.~d'Aspremont, L.~Ghaoui, M.~Jordan, and G.~Lanckriet.
\newblock A direct formulation for sparse pca using semidefinite programming.
\newblock \emph{Advances in neural information processing systems}, 17, 2004.

\bibitem[d'Aspremont et~al.(2008)d'Aspremont, Bach, and El~Ghaoui]{d2008optimal}
A.~d'Aspremont, F.~Bach, and L.~El~Ghaoui.
\newblock Optimal solutions for sparse principal component analysis.
\newblock \emph{Journal of Machine Learning Research}, 9\penalty0 (7), 2008.

\bibitem[Deshp et~al.(2016)Deshp, Montanari, et~al.]{deshp2016sparse}
Y.~Deshp, A.~Montanari, et~al.
\newblock Sparse pca via covariance thresholding.
\newblock \emph{Journal of Machine Learning Research}, 17\penalty0 (141):\penalty0 1--41, 2016.

\bibitem[Dey et~al.(2022)Dey, Mazumder, and Wang]{dey2022using}
S.~S. Dey, R.~Mazumder, and G.~Wang.
\newblock Using l1-relaxation and integer programming to obtain dual bounds for sparse pca.
\newblock \emph{Operations Research}, 70\penalty0 (3):\penalty0 1914--1932, 2022.

\bibitem[Dey et~al.(2023)Dey, Molinaro, and Wang]{dey2023solving}
S.~S. Dey, M.~Molinaro, and G.~Wang.
\newblock Solving sparse principal component analysis with global support.
\newblock \emph{Mathematical Programming}, 199\penalty0 (1):\penalty0 421--459, 2023.

\bibitem[Diakonikolas et~al.(2017)Diakonikolas, Kamath, Kane, Li, Moitra, and Stewart]{diakonikolas2017being}
I.~Diakonikolas, G.~Kamath, D.~M. Kane, J.~Li, A.~Moitra, and A.~Stewart.
\newblock Being robust (in high dimensions) can be practical.
\newblock In \emph{International Conference on Machine Learning}, pages 999--1008. PMLR, 2017.

\bibitem[Diakonikolas et~al.(2019)Diakonikolas, Kamath, Kane, Li, Steinhardt, and Stewart]{diakonikolas2019sever}
I.~Diakonikolas, G.~Kamath, D.~Kane, J.~Li, J.~Steinhardt, and A.~Stewart.
\newblock Sever: A robust meta-algorithm for stochastic optimization.
\newblock In \emph{International Conference on Machine Learning}, pages 1596--1606. PMLR, 2019.

\bibitem[d'Orsi et~al.(2020)d'Orsi, Kothari, Novikov, and Steurer]{d2020sparse}
T.~d'Orsi, P.~K. Kothari, G.~Novikov, and D.~Steurer.
\newblock Sparse pca: algorithms, adversarial perturbations and certificates.
\newblock In \emph{2020 IEEE 61st Annual Symposium on Foundations of Computer Science (FOCS)}, pages 553--564. IEEE, 2020.

\bibitem[d'Orsi et~al.(2021)d'Orsi, Liu, Nasser, Novikov, Steurer, and Tiegel]{d2021consistent}
T.~d'Orsi, C.-H. Liu, R.~Nasser, G.~Novikov, D.~Steurer, and S.~Tiegel.
\newblock Consistent estimation for pca and sparse regression with oblivious outliers.
\newblock \emph{Advances in Neural Information Processing Systems}, 34:\penalty0 25427--25438, 2021.

\bibitem[d’Aspremont et~al.(2014)d’Aspremont, Bach, and Ghaoui]{d2014approximation}
A.~d’Aspremont, F.~Bach, and L.~E. Ghaoui.
\newblock Approximation bounds for sparse principal component analysis.
\newblock \emph{Mathematical Programming}, 148:\penalty0 89--110, 2014.

\bibitem[Gally and Pfetsch(2016)]{gally2016computing}
T.~Gally and M.~E. Pfetsch.
\newblock Computing restricted isometry constants via mixed-integer semidefinite programming.
\newblock \emph{preprint, submitted}, 2016.

\bibitem[Kim et~al.(2022)Kim, Tawarmalani, and Richard]{kim2022convexification}
J.~Kim, M.~Tawarmalani, and J.-P.~P. Richard.
\newblock Convexification of permutation-invariant sets and an application to sparse principal component analysis.
\newblock \emph{Mathematics of Operations Research}, 47\penalty0 (4):\penalty0 2547--2584, 2022.

\bibitem[Li and Xie(2024)]{li2024exact}
Y.~Li and W.~Xie.
\newblock Exact and approximation algorithms for sparse principal component analysis.
\newblock \emph{INFORMS Journal on Computing}, 2024.

\bibitem[Magdon-Ismail(2017)]{magdon2017np}
M.~Magdon-Ismail.
\newblock Np-hardness and inapproximability of sparse pca.
\newblock \emph{Information Processing Letters}, 126:\penalty0 35--38, 2017.

\bibitem[Novikov(2023)]{novikov2023sparse}
G.~Novikov.
\newblock Sparse pca beyond covariance thresholding.
\newblock In \emph{The Thirty Sixth Annual Conference on Learning Theory}, pages 4737--4776. PMLR, 2023.

\bibitem[VU and LEI(2013)]{vu2013minimax}
V.~Q. VU and J.~LEI.
\newblock Minimax sparse principal subspace estimation in high dimensions.
\newblock \emph{The Annals of Statistics}, 41\penalty0 (6):\penalty0 2905--2947, 2013.

\bibitem[Wolsey and Nemhauser(2014)]{wolsey2014integer}
L.~A. Wolsey and G.~L. Nemhauser.
\newblock \emph{Integer and combinatorial optimization}.
\newblock John Wiley \& Sons, 2014.

\bibitem[Yuan and Zhang(2013)]{yuan2013truncated}
X.-T. Yuan and T.~Zhang.
\newblock Truncated power method for sparse eigenvalue problems.
\newblock \emph{Journal of Machine Learning Research}, 14\penalty0 (4), 2013.

\bibitem[Zhang et~al.(2012)Zhang, d’Aspremont, and Ghaoui]{zhang2012sparse}
Y.~Zhang, A.~d’Aspremont, and L.~E. Ghaoui.
\newblock Sparse pca: Convex relaxations, algorithms and applications.
\newblock \emph{Handbook on Semidefinite, Conic and Polynomial Optimization}, pages 915--940, 2012.

\bibitem[Zhou et~al.(2019)Zhou, Kantarcioglu, and Xi]{zhou2019survey}
Y.~Zhou, M.~Kantarcioglu, and B.~Xi.
\newblock A survey of game theoretic approach for adversarial machine learning.
\newblock \emph{Wiley Interdisciplinary Reviews: Data Mining and Knowledge Discovery}, 9\penalty0 (3):\penalty0 e1259, 2019.

\end{thebibliography}

\newpage
\appendix
\section{Proofs in Section~\ref{sec:SPCA-NOAP-refor}}  

\subsection{Proof of Proposition~\ref{prop:inner-min-refor}} \label{app:inner-min-refor}

\textbf{Recall the Proposition~\ref{prop:inner-min-refor}:} For a fixed \& feasible $\bm{v} \in \mathcal{V}_k$, the inner minimization problem satisfies 
\begin{align*}
    \min_{\|\bm{E}\| \leq \rho} ~ \frac{1}{n} \|\bm{X} \bm{v} + \bm{E}\bm{v}\|_2^2 ~~ = ~~ \frac{1}{n} \| \bm{X}\bm{v} - \proj_{\mathcal{F}(\bm{v}; \|\cdot\|, \rho)} (\bm{X} \bm{v}) \|_2^2 ~~, 
\end{align*} 
where set $\mathcal{F}(\bm{v}; \|\cdot\|, \rho) := \left\{ \bm{E}\bm{v} \in \mathbb{R}^n ~|~ \|\bm{E}\| \leq \rho \right\}$ and $\proj_{\mathcal{F}}(\bm{x}) := \argmin_{\bm{u} \in \mathcal{F}} \| \bm{x} - \bm{u} \|_2^2$ projects any point $\bm{x}$ onto a given set $\mathcal{F}$.

\begin{proof}
Geometrically, this inner minimization problem is equivalent to projecting the fixed vector $\bm{X}\bm{v}$ onto the inner feasible set generated by the adversarial perturbations, i.e., $\mathcal{F}(\bm{v}; \|\cdot\|, \rho) := \left\{ \bm{E}\bm{v} \in \mathbb{R}^n ~|~ \|\bm{E}\| \leq \rho \right\}$. Thus, the inner minimization problem can be written as
\begin{align*}
    \min_{\|\bm{E}\| \leq \rho} ~~ \frac{1}{n} \|\bm{X} \bm{v} + \bm{E}\bm{v}\|_2^2 = \frac{1}{n} \| \bm{X}\bm{v} - \proj_{\mathcal{F}(\bm{v}; \|\cdot\|, \rho)} (\bm{X} \bm{v}) \|_2^2 
\end{align*}
with $\proj_{\mathcal{F}(\bm{v}; \|\cdot\|, \rho)} (\bm{X} \bm{v}) := \argmin_{\bm{u} \in \mathcal{F}(\bm{v}; \|\cdot\|, \rho)} \| \bm{X} \bm{v} - \bm{u} \|_2^2$ a projection operator. Moreover, since the given matrix norm $\|\cdot\|$ is absolutely homogeneous and sub-additive, we can further ensure that the resulting feasible set for adversarial perturbations $\mathcal{F}(\bm{v}; \|\cdot\|, \rho)$ is convex and compact. 
\end{proof}

\subsection{Proof of Proposition~\ref{prop:samplewise-set}}

\textbf{Recall the Proposition~\ref{prop:samplewise-set}:} For any $\bm{v} \in \mathcal{V}_k$,  the resulting feasible set for adversarial perturbations $\mathcal{F}(\bm{v};\|\cdot\|_{2 \rightarrow \infty}, \rho)$ satisfies 
\begin{align*}
    \mathcal{F}(\bm{v};\|\cdot\|_{2 \rightarrow \infty}, \rho) = & ~ \left\{ 
    \bm{E} \bm{v} \in \mathbb{R}^n ~|~ \|\bm{E}_{i,:}\|_2 \leq \rho ~~ \forall ~ i \in [n] 
    \right\} \\
    = & ~ \left\{ \bm{u} \in \mathbb{R}^n ~|~ \|\bm{u}\|_{\infty} \leq \rho \|\bm{v}\|_2 \right\} . 
\end{align*}

\begin{proof}
In one direction, given $\bm{v} \in \mathcal{V}_k$, consider any point 
\begin{align*}
    \bm{E} \bm{v} \in \left\{ \bm{E} \bm{v} \in \mathbb{R}^n ~|~ \|\bm{E}_{i,:}\|_2 \leq \rho ~\forall ~ i \in [n] \right\} ~,
\end{align*}
we have $|[\bm{E}\bm{v}]_i| = |\bm{E}_{i,:} \bm{v}| \leq \|\bm{E}_{i, :}\|_2 \|\bm{v}\|_2 \leq \rho \|\bm{v}\|_2$. Therefore, we have
\begin{align*}
    \left\{ \bm{E} \bm{v} \in \mathbb{R}^n ~|~ \|\bm{E}_{i,:}\|_2 \leq \rho ~\forall ~ i \in [n] \right\} \subseteq \left\{ \bm{u} \in \mathbb{R}^n ~|~ \|\bm{u}\|_{\infty} \leq \rho \|\bm{v}\|_2 \right\}. 
\end{align*}
On the other direction, consider any $\bm{u} \in \{ \bm{u} \in \mathbb{R}^n ~|~ \|\bm{u}\|_{\infty} \leq \rho \|\bm{v}\|_2 \}$, i.e., $|\bm{u}_i| \leq \rho \|\bm{v}\|_2$, we can always pick $\bm{E}_{i, :} = \bm{u}_i \bm{v} / \|\bm{v}\|_2^2$ with $\|\bm{E}_{i, :}\|_2 \leq \rho$ such that $\langle \bm{E}_{i, :}, \bm{v} \rangle = \bm{u}_i$ for all $i \in [n]$. Thus, 
\begin{align*}
    \left\{ \bm{E} \bm{v} \in \mathbb{R}^n ~|~ \|\bm{E}_{i,:}\|_2 \leq \rho ~\forall ~ i \in [n] \right\} \supseteq \left\{ \bm{u} \in \mathbb{R}^n ~|~ \|\bm{u}\|_{\infty} \leq \rho \|\bm{v}\|_2 \right\}. 
\end{align*}
Combining two directions together completes the proof. 
\end{proof}

\subsection{Proof of Theorem~\ref{thm:samplewise-spca-MIP}} \label{app:samplewise-spca-MIP}

\textbf{Recall the Theorem~\ref{thm:samplewise-spca-MIP}:} 
Formulation~\eqref{eq:samplewise-spca-1} can be approximated by the following MIP.
\begin{align*}
    \begin{array}{rlll}
        \ub_k^{2 \rightarrow \infty} := \max\limits_{\bm{v}, \phi, \bm{g}, \bm{\xi}, \eta} & \sum_{j = 1}^d \lambda_j \bm{\xi}_j  + \frac{1}{n} \sum_{i = 1}^n \phi^i \\
        \emph{s.t.}~ & \bm{v} \in \overline{\mathcal{V}}_k, (\bm{g}, \bm{\xi}, \eta) \in \PLU([d]) \\
        & \phi^i \leq \phi_{\rho} (\langle \bm{x}^i, \bm{v} \rangle) & ~ \forall ~ i \in [n] 
    \end{array} ~ ,
\end{align*}
where its optimal value $\ub_k^{2 \rightarrow \infty}$ satisfies $\opt_k^{2 \rightarrow \infty} \leq \ub_k^{2 \rightarrow \infty} \leq \opt_k^{2 \rightarrow \infty} + \frac{1}{4 N^2} \sum_{j = 1}^d \lambda_j $.

\begin{proof}
Do singular value decomposition on $\widehat{\bm{\Sigma}}$ such that $\widehat{\bm{\Sigma}} = \sum_{j = 1}^d \lambda_j \bm{v}_j \bm{v}_j^{\top}$ with $\lambda_1 \geq \cdots \geq \lambda_d$ and define $\bm{g}_j := \langle \bm{v}_j, \bm{v} \rangle$ for all $j \in [d]$, then the non-convex constraint can be represented as $\obj = \bm{v}^{\top} \widehat{\bm{\Sigma}} \bm{v} = \sum_{j = 1}^d \lambda_j \bm{g}_j^2$. Using special-ordered set type-II (SOS-II) constraint, we can upper approximate the square term $\bm{g}_j^2$ based on the following procedures: 
\begin{enumerate}
    \item For each $j \in [d]$, let $\{\theta_j^{\ell} := \frac{\ell}{N} \}_{\ell = -N}^N$ be a sequence of splitting points that evenly splits the interval $[-1, 1]$, and let $\{\eta_j^{\ell} \}_{\ell = -N}^N$ a sequence of Special Ordered Set -- Type II (SOS-II) variables with respect to every aforementioned splitting point and also ensures $\sum_{\ell = -N}^N \eta_j^{\ell} = 1, \eta_j^{\ell} \geq 0$ for all $\ell$. 
    \item Based on the above splitting points and SOS-II variables, the decision variable $\bm{g}_j$ and its square $\bm{g}_j^2$ can be approximated by 
    \begin{align*}
        & \bm{g}_j := \sum_{\ell = - N}^N \theta_j^{\ell} \cdot \eta_j^{\ell} & \bm{\xi}_j := \sum_{\ell = - N}^N (\theta_j^{\ell})^2 \cdot \eta_j^{\ell} ~\in~ \left[ \bm{g}_j^2, ~ \bm{g}_j^2 +  \frac{1}{4 N^2} \right] .
    \end{align*}
    Refer to \PLU~\eqref{eq:PLU-set} for details. 
    
    \item Then, the non-convex constraint $\obj = \bm{v}^{\top} \widehat{\bm{\Sigma}} \bm{v} = \sum_{j = 1}^d \lambda_j \bm{g}_j^2$ can be approximated by a piece-wise linear function using SOS-II variables  
    \begin{align*}
        \sum_{j = 1}^d \lambda_j \bm{\xi}_j ~\in ~ & ~ \left[ \sum_{j = 1}^d \lambda_j \bm{g}_j^2, ~~  \sum_{j = 1}^d \lambda_j \bm{g}_j^2 + \frac{1}{4 N^2} \sum_{j = 1}^d \lambda_j \right] ~=~ \left[ \obj, ~~  \obj + \frac{1}{4 N^2} \sum_{j = 1}^d \lambda_j \right] ~~. 
    \end{align*}
\end{enumerate}
Therefore, the above maximization problem can be further approximated by a computationally tractable convex integer programming as follows: 
\begin{align*}
    \begin{array}{rlll}
        \ub_k^{2 \rightarrow \infty} := \max\limits_{\bm{v}, \phi, \bm{g}, \bm{\xi}, \eta} & \sum_{j = 1}^d \lambda_j \bm{\xi}_j  + \frac{1}{n} \sum_{i = 1}^n \phi^i \\
        \emph{s.t.}~ & \bm{v} \in \overline{\mathcal{V}}_k, (\bm{g}, \bm{\xi}, \eta) \in \PLU([d]) \\
        & \phi^i \leq \phi_{\rho} (\langle \bm{x}^i, \bm{v} \rangle) & ~ \forall ~ i \in [n] 
    \end{array} ~ ,
\end{align*}
Based on the above procedure, the piece-wise linear approximation from SOS-II ensures that 
\begin{align*}
    \opt_k^{2 \rightarrow \infty} + \frac{1}{4 N^2} \sum_{j = 1}^d \lambda_j \geq \ub_k^{2 \rightarrow \infty} \geq \opt_k^{2 \rightarrow \infty} . 
\end{align*}
\end{proof}

\subsection{Proof of Proposition~\ref{prop:samplewise-spca-MIP-r}} \label{app:samplewise-spca-MIP-r}

\textbf{Recall the Proposition~\ref{prop:samplewise-spca-MIP-r}:}
Given a pre-determined integer $r \leq d$, one can further approximate the Formulation~\eqref{eq:samplewise-spca-1} by the following mixed
integer convex program.
\begin{align*}
    \begin{array}{rlll}
        \ub_k^{2 \rightarrow \infty}(r) := \max\limits_{\bm{v}, \phi, \bm{g}, \bm{\xi}, \eta, \gamma} & \sum_{j = 1}^r \lambda_j \bm{\xi}_j  + \lambda_{r + 1} \gamma  + \frac{1}{n} \sum_{i = 1}^n \phi^i \\
        \emph{s.t.} ~~& \bm{v} \in \overline{\mathcal{V}}_k, (\bm{g}_j, \bm{\xi}_j, \eta_j)_{j \in [r]} \in \PLU([r])\\
        & \sum_{j = 1}^r \bm{g}_j^2 \leq 1 - \gamma, ~ \gamma \geq 0\\
        & \phi^i \leq \phi_{\rho} (\langle \bm{x}^i, \bm{v} \rangle) & ~ \forall ~ i \in [n] 
    \end{array} ~ . 
\end{align*}
The optimal value $\ub^{2 \rightarrow \infty}(r)$ satisfies $\opt^{2 \rightarrow \infty}
\leq \ub^{2 \rightarrow \infty}(r)\leq \opt^{2 \rightarrow \infty} + \frac{1}{4N^2} \sum_{j = 1}^r \lambda_j + \widehat{\gamma} (\lambda_{r + 1} - \lambda_d)$ with $\widehat{\gamma}$ being the optimal solution of the above convex integer programming.

\begin{proof}
Suppose $\bm{v}_*$ is the optimal solution for \eqref{eq:samplewise-spca-1}. Let $R := \text{span}(\bm{v}_1, \ldots, \bm{v}_r)$ for some $r \leq d$. 
Orthogonal decompose $\bm{v}_*$ such that $\bm{v}_* = \alpha \bm{v}_R + \beta \bm{v}_{\perp}$ with $\bm{v}_R \in R, ~ \bm{v}_{\perp} \perp R, ~ \|\bm{v}_R\|_2 = \|\bm{v}_{\perp}\|_2 = 1$, and $\alpha^2 + \beta^2 = 1$. Then, we have 
\begin{align*}
    \bm{v}_*^{\top} \widehat{\bm{\Sigma}} \bm{v}_* = \sum_{j = 1}^r \lambda_j \bm{g}_j^2 + \sum_{j > r} \lambda_j \bm{g}_j^2 \quad \text{with} \quad \sum_{j = 1}^r \bm{g}_j^2 = \alpha^2, \quad \sum_{j > r} \bm{g}_j^2 = \beta^2 , 
\end{align*}
and thus 
\begin{align*}
    \bm{v}_*^{\top} \widehat{\bm{\Sigma}} \bm{v}_* = \sum_{j = 1}^r \lambda_j \bm{g}_j^2 + \sum_{j > r} \lambda_j \bm{g}_j^2 \leq \sum_{j = 1}^r \lambda_j \bm{g}_j^2 + \beta^2 \lambda_{r + 1} \leq \sum_{j = 1}^r \lambda_j \bm{\xi}_j + \beta^2 \lambda_{r + 1}, 
\end{align*}
where the final inequality requests $r \times (2 N + 1)$ SOS-II variables. Since $\beta^2$ is unknown, in our new formulation, we can introduce a new variable $\gamma \in [0, 1]$ to denote $\beta^2$. Therefore, for a pre-determined parameter $r \leq d$, the maximization problem can be further written as  
\begin{align*}
    \begin{array}{rlll}
        \ub_k^{2 \rightarrow \infty}(r) := \max\limits_{\bm{v}, \phi, \bm{g}, \bm{\xi}, \eta, \gamma} & \sum_{j = 1}^r \lambda_j \bm{\xi}_j  + \lambda_{r + 1} \gamma  + \frac{1}{n} \sum_{i = 1}^n \phi^i \\
        \emph{s.t.} ~~& \bm{v} \in \overline{\mathcal{V}}_k, (\bm{g}_j, \bm{\xi}_j, \eta_j)_{j \in [r]} \in \PLU([r])\\
        & \sum_{j = 1}^r \bm{g}_j^2 \leq 1 - \gamma, ~ \gamma \geq 0\\
        & \phi^i \leq \phi_{\rho} (\langle \bm{x}^i, \bm{v} \rangle) & ~ \forall ~ i \in [n] 
    \end{array} ~ . 
\end{align*}
Similarly, the piece-wise linear approximation from SOS-II ensures that 
\begin{align*}
    \opt^{2 \rightarrow \infty} + \frac{1}{4N^2} \sum_{j = 1}^r \lambda_j + \widehat{\gamma} (\lambda_{r + 1} - \lambda_d) \geq \ub^{2 \rightarrow \infty}(r) \geq \opt^{2 \rightarrow \infty}
\end{align*}
with $\widehat{\gamma}$ the optimal solution of the above convex integer programming. 
\end{proof}

\subsection{Proof of Proposition~\ref{prop:featurewise-set}}

\textbf{Recall the Proposition~\ref{prop:featurewise-set}:}
For any $\bm{v} \in \mathcal{V}_k$, the feasible set for adversarial perturbations $\mathcal{F}(\bm{v};\|\cdot\|_{1 \rightarrow 2}, \rho)$ is 
\begin{align*}
    \mathcal{F}(\bm{v};\|\cdot\|_{1 \rightarrow 2}, \rho) := & ~ \left\{ 
    \bm{E} \bm{v} \in \mathbb{R}^n ~|~ \|\bm{E}_{:,j}\|_2 \leq \rho ~~ \forall ~ j \in [d] 
    \right\} \\
    = & ~ \left\{ \bm{u} \in \mathbb{R}^n ~|~ \|\bm{u}\|_2 \leq \rho \|\bm{v}\|_1 \right\} 
\end{align*}

\begin{proof}
Similar to the proof of Proposition~\ref{prop:featurewise-set}, it is easy to verify that 
\begin{align*}
    \left\{ 
    \bm{E} \bm{v} \in \mathbb{R}^n ~|~ \|\bm{E}_{:,j}\|_2 \leq \rho ~~ \forall ~ j \in [d] 
    \right\} \subseteq  \left\{ \bm{u} \in \mathbb{R}^n ~|~ \|\bm{u}\|_2 \leq \rho \|\bm{v}\|_1 \right\}  ~ . 
\end{align*}
On the other direction, consider any $\bm{u} \in \{ \bm{u} \in \mathbb{R}^n ~|~ \|\bm{u}\|_2 \leq \rho \|\bm{v}\|_1 \}$, by setting 
\begin{align*}
    \bm{E}_{:,j} = \frac{\bm{u}}{\|\bm{v}\|_1} \cdot {\tt sgn} \left(\bm{v}_j\right)~~ \forall ~ j \in [d] 
\end{align*}
with $\|\bm{E}_{:,j}\|_2 \leq \rho$, we have $\bm{E}\bm{v} = \bm{u}$. Thus, 
\begin{align*}
    \left\{ 
    \bm{E} \bm{v} \in \mathbb{R}^n ~|~ \|\bm{E}_{:,j}\|_2 \leq \rho ~~ \forall ~ j \in [d] 
    \right\} \supseteq  \left\{ \bm{u} \in \mathbb{R}^n ~|~ \|\bm{u}\|_2 \leq \rho \|\bm{v}\|_1 \right\}  ~ . 
\end{align*}
Combining two directions together completes the proof. 
\end{proof}

\subsection{Proof of Theorem~\ref{thm:featurewise-spca-MIP}} \label{app:featurewise-spca-MIP}

\textbf{Recall the Theorem~\ref{thm:featurewise-spca-MIP}:}
Formulation~\eqref{eq:emp-adv-spca-2} can be approximated by the following MIP. 
\begin{align*}
    \begin{array}{rlll}
        \sqrt{\ub^{1 \rightarrow 2}_k} := \max\limits_{\bm{v}, t, \bm{y}, \bm{g}, \bm{\xi}, \eta} & \frac{1}{\sqrt{n}} t \\
        \emph{s.t.}~~ & \bm{v} \in \overline{\mathcal{V}}_k, (\bm{g}, \bm{\xi}, \eta) \in \PLU([d]) \\
        & n \sum_{j = 1}^d \lambda_j \bm{\xi}_j \geq (t + \rho \bm{y})^2 \\
        & t \geq 0, ~ \bm{y} \geq \|\bm{v}\|_1
    \end{array} ~. 
\end{align*}
Moreover, $ \opt^{1 \rightarrow 2}_k\leq \ub^{1 \rightarrow 2}_k \leq \opt^{1 \rightarrow 2}_k + \frac{1}{4 N^2} \sum_{j = 1}^d \lambda_j  $.

\begin{proof}

By introducing a new non-negative variable $t$ to denote the term $\|\bm{X}\bm{v}\|_2 - \rho \|\bm{v}\|_1$, the problem~\eqref{eq:emp-adv-spca-2} becomes 
\begin{align*}
    & ~ \max_{\bm{v}, t} ~ \frac{1}{n} t^2 ~~\text{s.t.}~~ \|\bm{X}\bm{v}\|_2 - \rho \|\bm{v}\|_1 \geq t \geq 0, ~ \bm{v} \in \overline{\mathcal{V}}_k, \\
    \Leftrightarrow ~&~ \max_{\bm{v}, t, \bm{y}} ~ \frac{1}{n} t^2 ~~\text{s.t.}~~ \|\bm{X}\bm{v}\|_2 - \rho \bm{y} \geq t \geq 0, ~ \bm{y} \geq \|\bm{v}\|_1, ~ \bm{v} \in \overline{\mathcal{V}}_k, \\
    \Leftrightarrow ~&~ \max_{\bm{v}, t, \bm{y}} ~ \frac{1}{n} t^2 ~~\text{s.t.}~~ n \bm{v}^{\top} \bm{\Sigma} \bm{v} \geq (t + \rho \bm{y})^2 , ~ t \geq 0, ~ \bm{y} \geq \|\bm{v}\|_1, ~ \bm{v} \in \overline{\mathcal{V}}_k, 
\end{align*}
where the first if and only if holds since $\bm{y} = \|\bm{v}\|_1$ in optimality condition, the second if and only if holds since both $\bm{v}^{\top} \bm{\Sigma} \bm{v}$ and $t + \rho \bm{y}$ are non-negative. Recall $\bm{\Sigma} = \frac{1}{n} \bm{X}^{\top} \bm{X} = \frac{1}{n} \sum_{i = 1}^n \bm{x}^i (\bm{x}^i)^{\top}$ and do singular value decomposition on $\bm{\Sigma}$ such that  $\bm{\Sigma} = \sum_{j = 1}^d \lambda_j \bm{v}_j \bm{v}_j^{\top}$. The quadratic term $\bm{v}^{\top} \bm{\Sigma} \bm{v}$ becomes $\bm{v}^{\top} \bm{\Sigma} \bm{v} = \sum_{j = 1}^d \lambda_j \langle \bm{v}_j, \bm{v} \rangle^2 = \sum_{j = 1}^d \lambda_j \bm{g}_j^2$ with $\bm{g}_j = \langle \bm{v}_j, \bm{v} \rangle$ for all $j \in [d]$. Still using SOS-II constraint, one can upper approximate the quadratic term $\bm{g}_j^2$ via piecewise linear function as follows: for each $j \in [d]$, let $\{\theta_j^{\ell} := \frac{\ell}{N} \}_{\ell = -N}^N$ be a sequence splitting points that evenly splits the interval $[-1, 1]$, and let $\{\eta_j^{\ell} \}_{\ell = -N}^N$ a sequence of Special Ordered Set -- Type II (SOS-II) variables with respect to every aforementioned splitting point. Then for every $j \in [d]$, we have $\bm{g}_j$ and $\bm{g}_j^2$ can be approximated by 
\begin{align*}
    & \bm{g}_j := \sum_{\ell = - N}^N \theta_j^{\ell} \cdot \eta_j^{\ell} & \bm{\xi}_j := \sum_{\ell = - N}^N (\theta_j^{\ell})^2 \cdot \eta_j^{\ell} ~\in~ \left[ \bm{g}_j^2, ~ \bm{g}_j^2 +  \frac{1}{4 N^2} \right]
\end{align*}
and therefore the quadratic term $n \bm{v}^{\top} \bm{\Sigma} \bm{v}$  can be approximated using piece-wise linear function, 
\begin{align*}
    n \sum_{j = 1}^d \lambda_j \bm{\xi}_j - \frac{n}{4 N^2} \sum_{j = 1}^d \lambda_j \leq n \bm{v}^{\top} \bm{\Sigma} \bm{v} \leq n \sum_{j = 1}^d \lambda_j \bm{\xi}_j ~~ .
\end{align*}
To obtain an upper approximation for \eqref{eq:emp-adv-spca-2}, the original non-convex constraint can be therefore replaced by the convex constraint $n \sum_{j = 1}^d \lambda_j \bm{\xi}_j \geq (t + \rho \bm{y})^2$. Thus Formulation~\eqref{eq:emp-adv-spca-2} can be written as 
\begin{align*}
    \begin{array}{rlll}
        \sqrt{\ub^{1 \rightarrow 2}_k} := \max\limits_{\bm{v}, t, \bm{y}, \bm{g}, \bm{\xi}, \eta} & \frac{1}{\sqrt{n}} t \\
        \emph{s.t.}~~ & \bm{v} \in \overline{\mathcal{V}}_k, (\bm{g}, \bm{\xi}, \eta) \in \PLU([d]) \\
        & n \sum_{j = 1}^d \lambda_j \bm{\xi}_j \geq (t + \rho \bm{y})^2 \\
        & t \geq 0, ~ \bm{y} \geq \|\bm{v}\|_1
    \end{array} ~~,  
\end{align*}
which ensures $\opt^{1 \rightarrow 2}_k + \frac{1}{4 N^2} \sum_{j = 1}^d \lambda_j  \geq \ub^{1 \rightarrow 2}_k \geq \opt^{1 \rightarrow 2}_k$. 

\end{proof}

\subsection{Proof of Proposition~\ref{prop:featurewise-spca-MIP-r}} \label{app:featurewise-spca-MIP-r}

\begin{proposition}\label{prop:featurewise-spca-MIP-r}
For a pre-determined integer $r \leq d$, Formulation~\eqref{eq:emp-adv-spca-2} can be further approximated by the following MIP.
\begin{align}
    \begin{array}{rlll}
        \sqrt{\overline{\ub}^{1 \rightarrow 2}_k(r)} := \max\limits_{\bm{v}, t, \bm{y}, \bm{g}, \bm{\xi}, \eta, \gamma} & \frac{1}{\sqrt{n}} t \\
        \emph{s.t.} ~~~~~& \bm{v} \in \overline{\mathcal{V}}_k \\
        & \sum_{j = 1}^r \bm{g}_j^2 \leq 1 - \gamma, ~ \gamma \geq 0 \\
        & (\bm{g}_j, \bm{\xi}_j, \eta_j)_{j \in [r]} \in \PLU([r]) \\
        & n \sum_{j = 1}^r \lambda_j \bm{\xi}_j + n \lambda_{r + 1} \gamma \geq (t + \rho \bm{y})^2 \\
        & t \geq 0, ~ \bm{y} \geq \|\bm{v}\|_1 \\
    \end{array} ~,~ \label{eq:featurewise-MIP-r} 
\end{align}
which ensures $\ub^{1 \rightarrow 2}_k(r) 
\leq \overline{\ub}^{1 \rightarrow 2}_k(r) 
\leq \ub^{1 \rightarrow 2}_k(r) + \frac{\overline{t}^2}{4 \overline{N}^2} 
\leq \opt_k^{1 \rightarrow 2} + \frac{1}{4 N^2} \sum_{j = 1}^r \lambda_j + \widehat{\gamma} (\lambda_{r + 1} - \lambda_d) + \frac{\overline{t}^2}{4 \overline{N}^2}.$
\end{proposition}

\begin{proof}
For a fixed $r \leq d$, the quadratic term can be represented by 
\begin{align*}
    \bm{v}^{\top} \bm{\Sigma} \bm{v} = \sum_{j = 1}^r \lambda_j \bm{g}_j^2 + \sum_{j > r} \lambda_j \bm{g}_j^2 \leq \sum_{j = 1}^r \lambda_j \bm{g}_j^2 + \beta^2 \lambda_{r + 1} \leq \sum_{j = 1}^r \lambda_j \bm{\xi}_j + \beta^2 \lambda_{r + 1} ~ ~ , 
\end{align*}
where $\beta = \langle \bm{v}, \text{span}(\bm{v}_1, \ldots, \bm{v}_r) \rangle$ with $r \times (2N + 1)$ SOS-II variables. Then the above problem can be further written as 
\begin{align*}
    \begin{array}{rlll}
        \sqrt{\overline{\ub}^{1 \rightarrow 2}_k(r)} := \max\limits_{\bm{v}, t, \bm{y}, \bm{g}, \bm{\xi}, \eta, \gamma} & \frac{1}{\sqrt{n}} t \\
        \emph{s.t.} ~~~~~& \bm{v} \in \overline{\mathcal{V}}_k \\
        & \sum_{j = 1}^r \bm{g}_j^2 \leq 1 - \gamma, ~ \gamma \geq 0 \\
        & (\bm{g}_j, \bm{\xi}_j, \eta_j)_{j \in [r]} \in \PLU([r]) \\
        & n \sum_{j = 1}^r \lambda_j \bm{\xi}_j + n \lambda_{r + 1} \gamma \geq (t + \rho \bm{y})^2 \\
        & t \geq 0, ~ \bm{y} \geq \|\bm{v}\|_1 \\
    \end{array} ~,
\end{align*}
which ensures 
\begin{align*}
    \opt_k^{1 \rightarrow 2} + \frac{1}{4 N^2} \sum_{j = 1}^r \lambda_j + \widehat{\gamma} (\lambda_{r + 1} - \lambda_d) + \frac{\overline{t}^2}{4 \overline{N}^2} &\geq \ub^{1 \rightarrow 2}_k(r) + \frac{\overline{t}^2}{4 \overline{N}^2} \geq \overline{\ub}^{1 \rightarrow 2}_k(r)\\ &\geq \ub^{1 \rightarrow 2}_k(r) \geq \opt_k^{1 \rightarrow 2} .
\end{align*}
\end{proof}

\section{Proofs in Section~\ref{sec:stat-results}}

\subsection{Proof of Theorem~\ref{thm:pop-SPCA-NOAP-1}} \label{app:pop-SPCA-NOAP-1}

\textbf{Recall the Theorem~\ref{thm:pop-SPCA-NOAP-1}:}
The optimal solution for the population version of sparse PCA-NOAP~\eqref{eq:pop-adv-spca-1} equals the optimal solution for the population version of sparse PCA, i.e., 
\begin{align*}
    \argmax_{\|\bm{v}\|_2 = 1, \|\bm{v}\|_0 \leq k} ~ \mathbb{E}_{\bm{x} \sim \mathcal{N}(\bm{0}, \bm{\Sigma})} [\ell_{\rho}(\bm{v}^{\top} \bm{x})] \quad = \quad  \argmax_{\|\bm{v}\|_2 = 1, \|\bm{v}\|_0 \leq k} \bm{v}^{\top} \bm{\Sigma} \bm{v}.  
\end{align*} 

\begin{proof}
Let us first define the following events:
\begin{align*}
    \mathcal{E}(S_0) := \{ \bm{v}^{\top} \bm{x} \in [- \rho, ~ \rho] \}, \quad 
    \mathcal{E}(S_-) :=  \{ \bm{v}^{\top} \bm{x} < - \rho \}, \quad 
    \mathcal{E}(S_+) := \{ \bm{v}^{\top} \bm{x} > \rho \} . 
\end{align*}
Then we can represent the expected objective of \eqref{eq:pop-adv-spca-1} as a summation of three parts 
\begin{align*}
    \mathbb{E}_{\bm{x} \sim \mathcal{N}(\bm{0}, \bm{\Sigma})} [\ell_{\rho}(\bm{v}^{\top} \bm{x})] = & ~ \mathbb{E}_{\bm{x}} [0|\mathcal{E}(S_0)] \cdot \mathbb{P}(\mathcal{E}(S_0)) \\
    & ~ + \mathbb{E}_{\bm{x}} [(\bm{v}^{\top} \bm{x} - \rho)^2|\mathcal{E}(S_+)] \cdot \mathbb{P}(\mathcal{E}(S_+)) \\
    & ~ + \mathbb{E}_{\bm{x}} [(\bm{v}^{\top} \bm{x} + \rho)^2|\mathcal{E}(S_-)] \cdot \mathbb{P}(\mathcal{E}(S_-)) ~ . 
\end{align*}
Based on the symmetric property of Gaussian distribution $\mathcal{N}(\bm{0}, \bm{\Sigma})$, we have $\mathbb{E}_{\bm{x}} [(\bm{v}^{\top} \bm{x} - \rho)^2|\mathcal{E}(S_+)] \cdot \mathbb{P}(\mathcal{E}(S_+)) = \mathbb{E}_{\bm{x}} [(\bm{v}^{\top} \bm{x} + \rho)^2|\mathcal{E}(S_-)] \cdot \mathbb{P}(\mathcal{E}(S_-))$. Given a fixed $\bm{v}$, let $\bm{v}^{\top} \bm{x}_i =: x_i^v \sim \mathcal{N}(0, \sigma_v^2)$ an i.i.d. random variable with $\sigma_v^2 = \bm{v}^{\top} \bm{\Sigma} \bm{v}$, then set 
\begin{align*}
    h_{\rho} (\sigma_v) = & ~ \mathbb{E}_{\bm{x}} [(\bm{v}^{\top} \bm{x} - \rho)^2|\mathcal{E}(S_+)] \cdot \mathbb{P}(\mathcal{E}(S_+)) \\
    = & ~ \int_{x^v \geq \rho} (x^v - \rho)^2 f_{\mathcal{E}(S_+)} (x^v) \mathrm{d}x^v \cdot \mathbb{P}(\mathcal{E}(S_+)) \\
    = & ~ \int_{x^v \geq \rho} (x^v - \rho)^2 \frac{1}{\sigma_v \sqrt{2 \pi}} \exp \left( - \frac{(x^v)^2}{2 \sigma_v^2}  \right) \mathrm{d}x^v. 
\end{align*}
Taking gradient on $h_{\rho} (\sigma_v)$ with respect to $\sigma_v$ implies 
\begin{align*}
    \frac{\partial h_{\rho}}{\partial \sigma_v} = \frac{1}{\sqrt{2 \pi}} \bigg[ & ~ - \frac{1}{\sigma_v^2} \int_{x^v \geq \rho} (x^v - \rho)^2 \exp \left( - \frac{(x^v)^2}{2 \sigma_v^2}  \right) \mathrm{d}x^v \\
    & ~ + \frac{1}{\sigma_v^4} \int_{x^v \geq \rho} (x^v)^2 (x^v - \rho)^2 \exp \left( - \frac{(x^v)^2}{2 \sigma_v^2}  \right) \mathrm{d}x^v \bigg]. 
\end{align*}
Since 
\begin{align*}
    & ~ \int_{x^v \geq \rho} (x^v - \rho)^2 \exp \left( - \frac{(x^v)^2}{2 \sigma_v^2}  \right) \mathrm{d}x^v \\
    = & ~ \sqrt{\frac{\pi}{2}} \sigma_v (\rho^2 + \sigma_v^2) \cdot \text{erfc}\left( \frac{\rho}{\sqrt{2} \sigma_v} \right) - \rho \sigma_v^2 \exp \left( - \frac{\rho^2}{2 \sigma_v^2} \right) \\
    & ~ \int_{x^v \geq \rho} (x^v)^2 (x^v - \rho)^2 \exp \left( - \frac{(x^v)^2}{2 \sigma_v^2}  \right) \mathrm{d}x^v \\
    = & ~ \sqrt{\frac{\pi}{2}} \sigma_v^3 (\rho^2 + 3 \sigma_v^2) \cdot \text{erfc}\left( \frac{\rho}{\sqrt{2} \sigma_v} \right) - \rho \sigma_v^4 \exp \left( - \frac{\rho^2}{2 \sigma_v^2} \right)
\end{align*}
with $\text{erfc}\left( z \right) := \frac{2}{\sqrt{\pi}} \int_z^{\infty} \exp(- t^2) \mathrm{d}t > 0$, we have
\begin{align*}
    \int_{x^v \geq \rho} (x^v)^2 (x^v - \rho)^2 \exp \left( - \frac{(x^v)^2}{2 \sigma_v^2}  \right) \mathrm{d}x^v > \sigma_v^2 \int_{x^v \geq \rho} (x^v - \rho)^2 \exp \left( - \frac{(x^v)^2}{2 \sigma_v^2}  \right) \mathrm{d}x^v, 
\end{align*}
and thus $\frac{\partial h_{\rho}}{\partial \sigma_v} > 0$ for all $\sigma_v > 0$. Therefore, optimizing the population version of robust sparse PCA is equivalent to optimizing the classical sparse PCA problem for any fixed positive $\rho$, i.e., 
\begin{align*}
    \argmax_{\|\bm{v}\|_2 = 1, \|\bm{v}\|_0 \leq k} ~ h_{\rho} (\bm{v}^{\top} \bm{\Sigma} \bm{v}) \quad = \quad  \argmax_{\|\bm{v}\|_2 = 1, \|\bm{v}\|_0 \leq k} \bm{v}^{\top} \bm{\Sigma} \bm{v}.  
\end{align*}
\end{proof}

\subsection{Proof of Proposition~\ref{prop:featurewise-spca-pop}} \label{app:featurewise-spca-pop}

\textbf{Recall the Proposition~\ref{prop:featurewise-spca-pop}:}
Given a $\rho > 0$, the population version of Formulation~\ref{eq:emp-adv-spca-2} is 
\begin{align*}
    \max_{\bm{v} \in \mathcal{F}\left(\frac{\rho}{\sqrt{n}} \right)} ~ & ~ \left( \sqrt{\bm{v}^{\top} \bm{\Sigma} \bm{v} } - \frac{\rho}{\sqrt{n}} \|\bm{v}\|_1 \right)^2 - 2  \frac{\rho}{\sqrt{n}} \|\bm{v}\|_1 \sqrt{\bm{v}^{\top} \bm{\Sigma} \bm{v}} \cdot O \left( \frac{1}{n} \right) 
\end{align*}
with feasible set
\begin{align*}
    \mathcal{F}\left(\frac{\rho}{\sqrt{n}} \right) := \left\{ \bm{v} \in \mathcal{V}_s^d ~\left|~ \sqrt{\bm{v}^{\top} \bm{\Sigma} \bm{v}} \cdot \left( 1 + O\left( \frac{1}{n} \right) \right) \geq \frac{\rho}{\sqrt{n}} \|\bm{v}\|_1 \right. \right\} ~~ . 
\end{align*}

\begin{proof}

For a given $\rho > 0$, consider the population version of the feasible set of \eqref{eq:emp-adv-spca-2}, 
\begin{align*}
    \mathcal{F}\left(\frac{\rho}{\sqrt{n}} \right) := \left\{ \bm{v} \in \mathcal{V}_s^d ~\left|~ \mathbb{E}_{\bm{X}} \left[\frac{1}{\sqrt{n}}\|\bm{X}\bm{v}\|_2 \right] \geq \frac{\rho}{\sqrt{n}} \|\bm{v}\|_1 \right. \right\} ~~ . 
\end{align*}
Let $w_i := \bm{x}_i^{\top} \bm{v} \sim \mathcal{N}(0, \sigma_v^2)$ with $\sigma_v^2 = \bm{v}^{\top} \bm{\Sigma} \bm{v} = \lambda \langle \bm{v}_*, \bm{v} \rangle^2 + 1$ for all $i \in [n]$, then 
\begin{align*}
    \mathbb{E}_{\bm{X}} \left[\frac{1}{\sqrt{n}}\|\bm{X}\bm{v}\|_2 \right] = \mathbb{E}_{\bm{X}} \left[\frac{1}{\sqrt{n}}\|\bm{w}\|_2 \right] = \sigma_v \frac{\sqrt{2}}{\sqrt{n}} \frac{\Gamma(\frac{n + 1}{2})}{\Gamma(\frac{n}{2})} = \sigma_v \cdot \left( 1 + O\left( \frac{1}{n} \right) \right), 
\end{align*}
where $\Gamma(z) = \int_0^{\infty} t^{z- 1} \exp(-t) \mathrm{d}t$ denotes the Gamma function for any $z > 0$. We claim that: the set $\mathcal{F}\left(\frac{\rho}{\sqrt{n}} \right)$ is non-empty if 
\begin{align*}
    \rho \leq (\lambda + 1) \cdot \left( 1 + O\left( \frac{1}{n} \right) \right) \frac{\sqrt{n}}{\sqrt{k}} ~~. 
\end{align*}
Moreover, for a given $\|\bm{v}\|_2 = 1$, if 
\begin{align*}
    \rho \leq (\lambda \langle \bm{v}_*, \bm{v} \rangle^2 + 1) \cdot \left( 1 + O\left( \frac{1}{n} \right) \right) \frac{\sqrt{n}}{\sqrt{k}} ~~,  
\end{align*}
then $\bm{v} \in \mathcal{F}\left(\frac{\rho}{\sqrt{n}} \right)$. Now consider the population version of the objective function of \eqref{eq:emp-adv-spca-2}, and suppose that $\rho$ is appropriately chosen such that $\mathcal{F}\left(\frac{\rho}{\sqrt{n}} \right) \neq \emptyset$, 
\begin{align*}
    \max_{\bm{v} \in \mathcal{F}\left(\frac{\rho}{\sqrt{n}} \right)} ~ & ~ \mathbb{E}_{\bm{X}} \left[ \frac{1}{n} \left( \|\bm{X}\bm{v}\|_2 - \rho \|\bm{v}\|_1 \right)^2 \right] \\
    = \max_{\bm{v} \in \mathcal{F}\left(\frac{\rho}{\sqrt{n}} \right)} ~ & ~ \bm{v}^{\top} \bm{\Sigma} \bm{v} - 2 \frac{\rho}{\sqrt{n}} \|\bm{v}\|_1 \cdot \sqrt{\bm{v}^{\top} \bm{\Sigma} \bm{v}} \left( 1 + O\left( \frac{1}{n} \right) \right) + \frac{\rho^2}{n} \|\bm{v}\|_1^2 \\
    = \max_{\bm{v} \in \mathcal{F}\left(\frac{\rho}{\sqrt{n}} \right)} ~ & ~ \left( \sqrt{\bm{v}^{\top} \bm{\Sigma} \bm{v} } - \frac{\rho}{\sqrt{n}} \|\bm{v}\|_1 \right)^2 - 2  \frac{\rho}{\sqrt{n}} \|\bm{v}\|_1 \sqrt{\bm{v}^{\top} \bm{\Sigma} \bm{v}} \cdot O \left( \frac{1}{n} \right).
\end{align*}
\end{proof}

\subsection{Proof of Theorem~\ref{thm:vector-recovery}} \label{app:vector-recovery}

We begin by introducing an existing statistical result (Theorem 4.1, \citep{novikov2023sparse}) on ground truth recovery under \emph{oblivious} adversarial perturbation, see Proposition~\ref{prop:existing-stat-result}.  

\begin{proposition} \label{prop:existing-stat-result}
\textbf{Restatement of existing result .} Let $n, d, k, t \in \mathbb{N}, \lambda > 0, \delta \in (0,0.1)$. Let sample set $\widetilde{\bm{X}} = \sqrt{\lambda} \bm{u} \bm{v}_*^{\top} + \bm{W} + \bm{E}$,
where $\bm{u} \sim \mathcal{N}(0,1)^n$, $\bm{v}_* \in \mathbb{R}^d$ a $k$-sparse unit vector, $\bm{W} \sim \mathcal{N}(0,1)^{n \times d}$ independent of $\bm{u}$ and $\bm{E}$ is the adversary matrix such that 
\begin{align*}
    \|\bm{E}\|_{1 \rightarrow 2} = \rho \leq \epsilon \min\{\sqrt{\lambda}, \lambda\} \sqrt{\frac{n}{k}}  \quad \text{and} \quad \epsilon \sqrt{\ln(1/\epsilon)} \leq \delta^6 \min\left\{ 1, \min\{\sqrt{\lambda}, \lambda\} \sqrt{\frac{n}{k}} \right\}.
\end{align*}
Suppose the following additional parameter regime conditions are satisfied 
\begin{align*}
    n \gtrsim k + \frac{t \ln^2 d}{\delta^4}, \quad k \gtrsim \frac{t \ln d}{\delta^2}, \quad \lambda \gtrsim \frac{k}{\delta^6 \sqrt{t n}} \sqrt{\ln \left( 2 + \frac{td}{k^2} \left( 1 + \frac{d}{n} \right) \right)} ~. 
\end{align*}
Then there exists an algorithm that, given $\bm{X}, k, t$ in time $n \cdot d^{O(t)}$ outputs a unit vector $\widehat{\bm{v}}$ such that with probability at least $1 - o(1)$ as $d \rightarrow \infty$, $|\langle \widehat{\bm{v}}, \bm{v}_* \rangle | \geq 1 - \delta$. 
\end{proposition}
Next, we will present our main theoretical results on ground truth recovery under adversarial perturbations and compare our main results to the existing result (Proposition~\ref{prop:existing-stat-result}) aforementioned. Before proving our main result (Theorem~\ref{thm:vector-recovery}), let us first show the following Lemma holds. 

\begin{lemma} \label{lemma:support-set} 
The support of optimal solution $\supp(\widehat{\bm{v}})$ is contained within $S_*$. 
\end{lemma}

\begin{proof}
We prove this by contradiction. Recall $\widehat{\bm{v}} := \argmax_{\bm{v} \in \mathcal{V}_k} ~ \sqrt{\bm{v}^{\top} \bm{\Sigma} \bm{v} } - \frac{\rho}{\sqrt{n}} \|\bm{v}\|_1$ be the optimal solution of Formulation~\eqref{eq:pop-featurewise-spca}. Suppose $\widehat{\bm{v}}$ has non-zero indexes outside $S$. WLOG, we have $\langle \widehat{\bm{v}}, \bm{v}_* \rangle \geq 0$, otherwise, $-\widehat{\bm{v}}$ is also an optimal solution and we take this solution instead. Suppose index $i \notin S$ and $\widehat{\bm{v}}_i \neq 0$. Since $|\supp(\widehat{\bm{v}})| \leq k$, $|\supp(\widehat{\bm{v}}) \cap S| \leq k-1$. Therefore, there exists an index $j \in S$ such that $\widehat{\bm{v}}_j = 0$. Define $\widehat{\bm{v}}'$ where 
\begin{align*}
    \left\{
    \begin{array}{llll}
        \widehat{\bm{v}}'_l = \widehat{\bm{v}}_l,  ~~\text{for} ~~ l \neq i,j  \\
        \widehat{\bm{v}}'_i = 0\\
        \widehat{\bm{v}}'_j = |\widehat{\bm{v}}_i| \times \text{sgn} \left({\bm{v}_{*j}} \right)
    \end{array}
    \right. ~~,
    \end{align*}
We have $\|\widehat{\bm{v}}'\|_2 = \|\widehat{\bm{v}}\|_2$, $\|\widehat{\bm{v}}'\|_1 = \|\widehat{\bm{v}}\|_1$ and $|\supp(\widehat{\bm{v}}')| = |\supp(\widehat{\bm{v}})| \leq k$, so $\widehat{\bm{v}}'$ is feasible.
\begin{align*}
    \langle \widehat{\bm{v}}', \bm{v}_* \rangle  = & ~ \sum_{l \in S, l \ \neq j} \widehat{\bm{v}}_l' {\bm{v}_*}_l + \widehat{\bm{v}}_j' {\bm{v}_*}_j\\
         > & ~ \sum_{l \in S, l \ \neq j} \widehat{\bm{v}}_l' {\bm{v}_*}_l = \sum_{l \in S, l \ \neq j} \widehat{\bm{v}}_l {\bm{v}_*}_l + \widehat{\bm{v}}_j {\bm{v}_*}_j = \langle \widehat{\bm{v}}, \bm{v}_* \rangle \geq 0
\end{align*}
Inserting the above inequality to the objective function implies 
\begin{align*}
    & ~ \sqrt{\widehat{\bm{v}}'^{\top} \bm{\Sigma} \widehat{\bm{v}}'} - \frac{\rho}{\sqrt{n}} \|\widehat{\bm{v}}'\|_1 = \sqrt{1 + \lambda \langle \widehat{\bm{v}}', \bm{v}_* \rangle^2} - \frac{\rho}{\sqrt{n}} \|\widehat{\bm{v}}'\|_1\\
    > & ~ \sqrt{1 + \lambda \langle \widehat{\bm{v}}, \bm{v}_* \rangle^2} - \frac{\rho}{\sqrt{n}} \|\widehat{\bm{v}}\|_1 = \sqrt{\widehat{\bm{v}}^{\top} \bm{\Sigma} \widehat{\bm{v}}} - \frac{\rho}{\sqrt{n}} \|\widehat{\bm{v}}\|_1
\end{align*}
    which contradict with the optimality of $\widehat{\bm{v}}$.
\end{proof}
Based on Lemma~\ref{lemma:support-set}, we are poised to present the proof of Theorem~\ref{prop:existing-stat-result}. \\

\noindent \textbf{Recall the Theorem~\ref{thm:vector-recovery}:} Assume samples $\bm{X}$ are i.i.d. generated based on Assumption~\ref{assump:sample-model} and parameter $\rho$ satisfies 
\begin{align*}
    \rho \leq c \cdot \delta \lambda \sqrt{\frac{n}{k}}. 
\end{align*}
Then the optimal solution 
\begin{align*}
    \widehat{\bm{v}} := \argmax_{\bm{v} \in \mathcal{V}_k} ~ \sqrt{\bm{v}^{\top} \bm{\Sigma} \bm{v} } - \frac{\rho}{\sqrt{n}} \|\bm{v}\|_1 
\end{align*}
ensures $|\langle \widehat{\bm{v}}, \bm{v}_* \rangle | \geq 1 - \delta$, where $\bm{v}_*$ denotes the optimal solution of the sparse PCA problem, i.e., $\bm{v}_* := \argmax_{\bm{v} \in \mathcal{V}_k} \bm{v}^{\top} \bm{\Sigma} \bm{v}$. 

\begin{proof}
Using the result of Lemma~\ref{lemma:support-set}, it is sufficient to restrict the constraint from $\bm{v} \in \mathcal{V}_k$ to $\|\bm{v}\|_2 = 1, \supp(\bm{v}) = S_*$. Here, we compute the upper bound of $\rho$ under which the optimal solution of 
\begin{align*}
    \widehat{\bm{v}} := \argmax_{\bm{v} \in \mathcal{V}_s^d} ~ \sqrt{\bm{v}^{\top} \bm{\Sigma} \bm{v} } - \frac{\rho}{\sqrt{n}} \|\bm{v}\|_1
\end{align*}
ensures $|\langle \widehat{\bm{v}}, \bm{v}_* \rangle | \geq 1 - \delta$. Show by contradiction, suppose $|\langle \widehat{\bm{v}}, \bm{v}_* \rangle | \leq 1 - \delta$ holds, the optimality of $\widehat{\bm{v}}$ implies  
\begin{align*}
    \sqrt{1 + \lambda(1 - \delta)^2 } - \frac{\rho}{\sqrt{n}} \|\widehat{\bm{v}}\|_1 \geq \sqrt{\widehat{\bm{v}}^{\top} \bm{\Sigma} \widehat{\bm{v}} } - \frac{\rho}{\sqrt{n}} \|\widehat{\bm{v}}\|_1 \geq \sqrt{1 + \lambda} - \frac{\rho}{\sqrt{n}} \|\bm{v}_*\|_1
\end{align*}
which implies
\begin{align*}
     \rho \geq \frac{\lambda (2 \delta - \delta^2)}{\sqrt{1 + \lambda} + \sqrt{1 + \lambda (1 - \delta)^2}} \cdot \frac{\sqrt{n}}{\|\bm{v}_*\|_1 - \|\widehat{\bm{v}}\|_1} ~~. 
\end{align*}
Let $\widehat{\bm{e}} := \bm{v}_* - \widehat{\bm{v}}$ with number of nonzero components at most $k$ by Lemma~\ref{lemma:support-set}, then
\begin{align*}
    \|\bm{v}_*\|_1 - \|\widehat{\bm{v}}\|_1 \leq \|\widehat{\bm{e}}\|_1 \leq \sqrt{\|\widehat{\bm{e}}\|_0} \|\widehat{\bm{e}}\|_2 \leq \sqrt{k} \|\widehat{\bm{e}}\|_2 \leq \sqrt{k} \sqrt{2 - 2 \langle \widehat{\bm{v}}, \bm{v}_* \rangle} ~~ . 
\end{align*}
Thus, we have
\begin{align*}
    \rho \geq & ~ \frac{\lambda (2 \delta - \delta^2)}{\sqrt{1 + \lambda} + \sqrt{1 + \lambda (1 - \delta)^2}} \cdot \frac{\sqrt{n}}{\|\bm{v}_*\|_1 - \|\widehat{\bm{v}}\|_1} \\
    \geq & ~ \frac{\lambda (2 \delta - \delta^2)}{\sqrt{1 + \lambda} + \sqrt{1 + \lambda (1 - \delta)^2}} \cdot \frac{\sqrt{n}}{ \sqrt{k} \cdot \max\{1, \sqrt{2}\sqrt{1 - \langle \widehat{\bm{v}}, \bm{v}_* \rangle} \} } \\
    = & ~ \underbrace{\frac{2 \delta - \delta^2}{\sqrt{6} \cdot \max\{1, \sqrt{2}\sqrt{1 - \langle \widehat{\bm{v}}, \bm{v}_* \rangle} \} }}_{=: \epsilon' } \cdot \lambda \cdot \frac{\sqrt{n}}{ \sqrt{k} } 
    =: \rho^*  ~~. 
\end{align*}
That is to say, if $\rho < \rho^*$, $|\langle \widehat{\bm{v}}, \bm{v}_* \rangle | \geq 1 - \delta$ holds. Compared with the existing statistical result, we have
\begin{align*}
    O(\delta) = \epsilon' \geq \epsilon = O(\delta^6) \quad \text{and} \quad \lambda \geq \min\{\sqrt{\lambda}, \lambda\} ~~, 
\end{align*}
therefore, the proposed MIP reformulation is more robust since it ensures a greater upper bound on $\|\bm{E}\|_{1 \rightarrow 2}$. 
\end{proof}

\subsection{Proof of Theorem~\ref{thm:strong-weak-signal}} \label{app:strong-weak-signal}

\textbf{Recall the Theorem~\ref{thm:strong-weak-signal}:}
Under Assumption~\ref{assump:strong-weak}, the behavior of $\widehat{\bm{v}}$ satisfies the following:

\noindent \textbf{1.}  \textbf{Recovery stage.} When $\rho \in [0, ~ O(\min\{\sqrt{\lambda}, \lambda\} \sqrt{ n / k_2})]$, $\widehat{\bm{v}}$ recovers the ground truth $\bm{v}_*$ defined in Theorem~\ref{thm:vector-recovery}. 

\noindent \textbf{2.} \textbf{Robust stage.} When $\rho \in (O(\min\{\sqrt{\lambda}, \lambda\} \sqrt{n / k_2}), ~ O(\min\{\sqrt{\lambda}, \lambda\} \sqrt{ n / k_1})]$, $\widehat{\bm{v}}$ recovers the strong signal part in $\bm{v}_*$ while eliminating the weak signal. 

\noindent \textbf{3.} \textbf{Overly perturbed stage.} When $\rho \in (O(\min\{\sqrt{\lambda}, \lambda\} \sqrt{ n / k_1}), ~ + \infty)$, adversarial perturbation is too large to recover enough information of the ground truth. 

\begin{proof}
Our analysis focuses on the perturbation parameter of the robust stage. The proof sketch mainly contains three steps. Step-1 provides a simplified reformulation of the objective function for Formulation~\eqref{eq:pop-featurewise-spca} under Assumption~\ref{assump:strong-weak}. Step-2 then demonstrates that such optimality of this reformulation can be achieved at only four distinct points. Consequently, determining optimality reduces to comparing objective values among these four points. Based on this comparison, we derive conditions (lower and upper bounds) for $\rho$, ensuring that one of these points is guaranteed to be optimal. Finally, building upon Step-2, Step-3 establishes the connections between adversarial perturbation parameter $\rho$ and the behavior of $\widehat{\bm{v}}$ as presented in Theorem~\ref{thm:strong-weak-signal}. \\

\noindent \textbf{Step 1} Based on the above Lemma~\ref{lemma:support-set}, define the optimal solution $\widehat{\bm{v}} = (\widehat{\bm{v}}_{S_1}, \widehat{\bm{v}}_{S_2}, \bm{0}_{[d] \backslash S})$ with
\begin{align*}
    \left\{
    \begin{array}{llll}
        \|\widehat{\bm{v}}_{S_1}\|_2 = \cos \theta, ~~ \|\widehat{\bm{v}}_{S_2}\|_2 = \sin \theta ~~\text{for} ~~ \theta \in [0, \pi/2]  \\
        \langle \widehat{\bm{v}}_{S_1}, \bm{v}_* \rangle = \sqrt{c} \cos \theta \cos \theta_1, ~~ \langle \widehat{\bm{v}}_{S_2}, \bm{v}_* \rangle = \sqrt{1 - c} \sin \theta \cos \theta_2
    \end{array}
    \right. ~~,
\end{align*}
where $\theta_1, \theta_2 \in [0, \pi/2]$ denote the angles between vectors $\widehat{\bm{v}}_{S_1}, \bm{v}_*$ and vectors $\widehat{\bm{v}}_{S_2}, \bm{v}_*$, respectively. Note that the components of $\widehat{\bm{v}}$ are all non-negative, otherwise, we can flip the sign of these components and we will obtain greater inner product of $\widehat{\bm{v}}$ and $\bm{v}_*$. Since the value of $\bm{v_*}$ on $S_1$ (or $S_2$) are all the same, we can directly derive the $\ell_1$-norm of $\widehat{\bm{v}}_{S_1}$ (or $\widehat{\bm{v}}_{S_1}$) from the corresponding inner product. Therefore,  the $\ell_1$-norm of $\bm{\widehat{v}}$ can be explicitly computed as follows.
\begin{align*}
    \|\widehat{\bm{v}}\|_1 = & ~ \|\widehat{\bm{v}}_{S_1}\|_1 + \|\widehat{\bm{v}}_{S_2}\|_1 = \frac{\langle \widehat{\bm{v}}_{S_1}, \bm{v}_* \rangle}{\sqrt{\frac{c}{k_1}}} + \frac{\langle \widehat{\bm{v}}_{S_2}, \bm{v}_* \rangle}{\sqrt{\frac{1-c}{k_2}}}
    = \sqrt{k_1} \cos \theta \cos \theta_1 + \sqrt{k_2} \sin \theta \cos \theta_2 ~ . 
\end{align*}
Based the setting of covariance matrix $\bm{\Sigma}$ in Assumption~\ref{assump:strong-weak}, plugging the above expression into the objective function gives
\begin{align*}
    \sqrt{\widehat{\bm{v}}^{\top} \bm{\Sigma} \widehat{\bm{v}} } - \frac{\rho}{\sqrt{n}} \|\widehat{\bm{v}}\|_1 =  & ~ \sqrt{\widehat{\bm{v}}^{\top} (\bm{I_d} + \lambda \bm{v_*}^{\top} \bm{v_*}) \widehat{\bm{v}} } - \frac{\rho}{\sqrt{n}} \|\widehat{\bm{v}}\|_1 \\
    = & ~ \sqrt{1 + \lambda \langle \widehat{\bm{v}}, \bm{v}_* \rangle^2 } - \frac{\rho}{\sqrt{n}} \|\widehat{\bm{v}}\|_1 \\
    = & ~ \sqrt{1 + \lambda \left(\langle \widehat{\bm{v}}_{S_1}, \bm{v}_* \rangle + \lambda \langle \widehat{\bm{v}}_{S_2}, \bm{v}_* \rangle \right)^2 } - \frac{\rho}{\sqrt{n}} \|\widehat{\bm{v}}\|_1 \\
    = & ~ \sqrt{1 + \lambda \left(\sqrt{c} \cos \theta \cos \theta_1 +  \sqrt{1 - c} \sin \theta \cos \theta_2 \right)^2} \\
    & ~ - \frac{\rho}{\sqrt{n}} \left( \sqrt{k_1} \cos \theta \cos \theta_1 + \sqrt{k_2} \sin \theta \cos \theta_2 \right) ~. 
\end{align*}
To be concise, define $x_1 := \cos \theta_1 \in [0,1]$, $x_2 := \cos \theta_2 \in [0,1]$, we can further translate the objective function into a triple-variable function $f(\theta, x_1, x_2)$ defined as follows
\begin{align*}
    f \left(\theta,x_1,x_2 \right) := & ~  \sqrt{1 + \lambda \left(\sqrt{c} \cos \theta x_1 +  \sqrt{1 - c} \sin \theta x_2 \right)^2} \\
    & ~ - \frac{\rho}{\sqrt{n}} \left( \sqrt{k_1} \cos \theta x_1 + \sqrt{k_2} \sin \theta x_2 \right) ~. 
\end{align*}

\noindent\textbf{Step 2.1} For a fixed $\theta$, we then search for the maximum of $f(\theta, x_1, x_2)$ over $x_1$ and $x_2$. Taking partial derivative over $x_1$, we have 
\begin{align*}
    \frac{\partial f }{\partial x_1} \left(\theta,x_1,x_2 \right) =& \frac{\lambda \left(\sqrt{c} \cos \theta x_1 + \sqrt{1-c}\sin \theta \cos \theta_2 \right) \left(\sqrt{c} \cos \theta \right)}{\sqrt{1 + \lambda \left(\sqrt{c} \cos \theta x + \sqrt{1-c}\sin \theta \cos \theta_2 \right)^2}} - \frac{\rho}{\sqrt{n}} \sqrt{k_1} \cos \theta\\
    =&\frac{\lambda \sqrt{c} \cos \theta}{\sqrt{\frac{1}{\left( \sqrt{c} \cos \theta x_1 + \sqrt{1-c}\sin \theta \cos \theta_2 \right)^2} + \lambda}} - \frac{\rho}{\sqrt{n}} \sqrt{k_1} \cos \theta
\end{align*}
This derivative increases monotonically with respect to $x_1$, meaning that the objective function is convex with regard to $x_1$, so the optimal value is attained at the boundary (extreme points), i.e. $x_1 = 0$ or $x_1 = 1$. Similarly, by taking partial derivative over $x_2$, it is easy to conclude that the optimal value is attained when $x_2 = 0$ or $x_2 = 1$. 
Thus, determining optimality of function $f(\theta, x_1, x_2)$, given $\theta$, reduces to comparing objective values among four points $\left( x_1, x_2 \right) = \left( 0, 0 \right), \left( 1, 0 \right), \left( 0, 1 \right), \left( 1, 1 \right)$, which gives the following four objective function values:
\begin{enumerate}
    \item $f \left(\theta, 0,0 \right) = 1$; 
    \item $f \left(\theta, 1,0 \right) = \sqrt{1 + \lambda c \cos^2 \theta} - \frac{\rho}{\sqrt{n}} \sqrt{k_1} \cos \theta$; 
    \item $f \left(\theta, 0,1 \right) = \sqrt{1 + \lambda \left(1 - c \right) \sin^2 \theta} - \frac{\rho}{\sqrt{n}} \sqrt{k_2} \sin \theta$; 
    \item $f \left(\theta, 1,1 \right) = \sqrt{1 + \lambda \left(\sqrt{c} \cos \theta + \sqrt{1-c} \sin \theta \right)^2} - \frac{\rho}{\sqrt{n}} \left(\sqrt{k_1} \cos \theta + \sqrt{k_2} \sin \theta \right)$ 
\end{enumerate}
and the original objective can be represented as 
\begin{align*}
    &\max_{\theta \in [0, \pi/2]} \left\{ f \left(\theta, 0,0 \right), f \left(\theta, 1,0 \right), f \left(\theta, 0,1 \right), f \left(\theta, 1,1 \right)\right\}\\
    = &\max \left\{ \max_{\theta \in [0, \pi/2]} f \left(\theta, 0,0 \right), \max_{\theta \in [0, \pi/2]} f \left(\theta, 1,0 \right), \max_{\theta \in [0, \pi/2]} f \left(\theta, 0,1 \right), \max_{\theta \in [0, \pi/2]} f \left(\theta, 1,1 \right)\right\}.  
\end{align*}
It remains to compare these four maximum values and pick out the largest one. Since our goal is to find the perturbation range of the robust stage, we would like to figure out under which condition (parameter regime) $f \left(0, 1,0 \right)$ takes the maximum. This is because when $x_1 = 1, x_2 = 0$, we have $\langle \widehat{\bm{v}}_{S_1}, \bm{v}_* \rangle = \sqrt{c}, \langle \widehat{\bm{v}}_{S_2}, \bm{v}_* \rangle = 0$, which represents the stage that the strong signal is preserved and the weak signal is eliminated. 

Since it is hard to find a closed-form representation of $\theta$ under which $f \left(\theta, 1,1 \right)$ achieves its maximum, so we first compare $f \left(\theta, 1,0 \right)$ and $f \left(\theta, 1,1 \right)$ for every $\theta$ to show that under some condition, $\max_{\theta \in [0, \pi/2]} f \left(\theta, 1,1 \right)$ cannot be the largest. For any $\theta \in [0, \pi/2]$, we have
\begin{align*}
    & ~ \quad f \left(\theta, 1,0 \right) \geq f \left(\theta, 1,1 \right) \\
    \Leftrightarrow ~&~ \quad \rho \geq \frac{\sqrt{1 + \lambda (\sqrt{c} \cos \theta + \sqrt{1-c} \sin \theta)^2} - \sqrt{1 + \lambda c \cos^2 \theta}}{\sin \theta} \frac{\sqrt{n}}{\sqrt{k_2}} ~. 
\end{align*}
Note that the above right-hand-side term can be upper-bounded by 
\begin{align*}
    & ~ \frac{\sqrt{1 + \lambda \left( \sqrt{c} \cos \theta + \sqrt{1-c} \sin \theta \right)^2} - \sqrt{1 + \lambda c \cos^2 \theta}}{\sin \theta} \frac{\sqrt{n}}{\sqrt{k_2}}\\
    = & ~ \frac{\lambda \left(2 \sqrt{c \left(1-c \right)} \sin \theta \cos \theta + \left( 1-c \right) \sin^2 \theta \right)}{\sin \theta \left(\sqrt{1 + \lambda (\sqrt{c} \cos \theta + \sqrt{1-c} \sin \theta)^2} + \sqrt{1 + \lambda c \cos^2 \theta} \right)} \frac{\sqrt{n}}{\sqrt{k_2}}\\
    = & ~ \frac{\lambda \left(2 \sqrt{c \left(1-c \right)} \cos \theta + \left(1-c \right) \sin \theta \right)}{\sqrt{1 + \lambda \left(\sqrt{c} \cos \theta + \sqrt{1-c} \sin \theta \right)^2} + \sqrt{1 + \lambda c \cos^2 \theta}} \frac{\sqrt{n}}{\sqrt{k_2}}\\
    \leq & ~ \frac{\sqrt{-3 c^2 + 2c + 1}}{\sqrt{1 + \lambda \left(1-c \right)} + 1} \frac{\lambda \sqrt{n}}{\sqrt{k_2}} = O \left(\frac{\lambda \sqrt{n}}{\sqrt{k_2}} \right) ~ . 
\end{align*}
Note that the last inequality holds independently on the choice of $\theta$, thus, we get a lower bound of order $O \left(\frac{\lambda \sqrt{n}}{\sqrt{k_2}} \right)$ for $\rho$ that guarantees $f \left(\theta, 1,0 \right) \geq f \left(\theta, 1,1 \right)$ regardless of $\theta$. In this case, the original objective is reduced to
\begin{align*}
    \max \left\{ \max_{\theta \in [0, \pi/2]} f \left(\theta, 0,0 \right), \max_{\theta \in [0, \pi/2]} f \left(\theta, 1,0 \right), \max_{\theta \in [0, \pi/2]} f \left(\theta, 0,1 \right) \right\}.  
\end{align*}

\noindent\textbf{Step 2.2} We then compute the maximum of the above three functions, respectively. The trivial case is $\max_{\theta \in [0, \pi/2]} f \left(\theta, 0,0 \right) = 1$. 

For $\max_{\theta \in [0, \pi/2]} f \left(\theta, 1,0 \right)$, our goal is to find out the range of $\rho$ such that $f(0,1,0)$ is optimal. By taking the partial derivative,
\begin{align*}
    \frac{\partial f }{\partial \theta} \left(\theta,1,0 \right) = & ~ - \frac{\lambda c \cos \theta \sin \theta}{\sqrt{1 + \lambda c \cos^2 \theta}} + \frac{\rho \sqrt{k_1}}{\sqrt{n}} \sin \theta \\
    = & ~ \sin \theta \left(\frac{\rho \sqrt{k_1}}{\sqrt{n}} - \frac{\lambda c \cos \theta}{\sqrt{1 + \lambda c \cos^2 \theta}} \right)\\
    = & ~ \sin \theta \left(\frac{\rho \sqrt{k_1}}{\sqrt{n}} - \frac{\lambda c}{\sqrt{\frac{1}{\cos^2 \theta} + \lambda c }} \right)
\end{align*}
When $\rho > \frac{\lambda c}{\sqrt{1 + \lambda c}} \frac{\sqrt{n}}{\sqrt{k_1}}$, the partial derivative is always non-negative, indicating that $f(\frac{\pi}{2},1,0) = 1$ is optimal, which is not the case we are interested in. When $\rho \leq \frac{\lambda c}{\sqrt{1 + \lambda c}} \frac{\sqrt{n}}{\sqrt{k_1}}$, the partial derivative is first negative then positive, meaning that the optimal point can only be achieved at $\theta = 0$ or $\theta = \frac{\pi}{2}$. If $\theta = 0$, $f(0,1,0) = \sqrt{1 + \lambda c} - \frac{\rho \sqrt{k_1}}{\sqrt{n}}$. If $\theta = \frac{\pi}{2}$, $f(\frac{\pi}{2},1,0) = 1$ holds trivially. Thus,
\begin{align*}
    f(0,1,0) \geq f(\frac{\pi}{2},1,0) = 1 \quad 
    \Leftrightarrow \quad \rho \leq \frac{\lambda c}{1 + \sqrt{1 + \lambda c}} \frac{\sqrt{n}}{\sqrt{k_1}}
\end{align*}
This inequality also directly imply that $f(0,1,0) \geq \max_{\theta \in [0, \frac{\pi}{2}]} f(\theta,0,0)$. Putting the above two upper bounds together gives the parameter regime of $\rho$ as $\rho \leq \frac{\lambda c}{1 + \sqrt{1 + \lambda c}} \frac{\sqrt{n}}{\sqrt{k_1}}$.\\

\noindent We then discuss $\max_{\theta \in [0, \pi/2]} f \left(\theta, 0,1 \right)$ under this parameter regime. By taking the partial derivative,
\begin{align*}
    \frac{\partial f }{\partial \theta} \left(\theta,0,1 \right) = & ~  \frac{\lambda (1-c) \cos \theta \sin \theta}{\sqrt{1 + \lambda (1-c) \sin^2 \theta}} - \frac{\rho \sqrt{k_2}}{\sqrt{n}} \cos \theta \\
    = & ~ \cos \theta \left( \frac{\lambda (1-c)}{\sqrt{\frac{1}{\sin^2 \theta} + \lambda (1-c)}} - \frac{\rho \sqrt{k_2}}{\sqrt{n}} \right)
\end{align*}
When $\rho > \frac{\lambda (1-c)}{\sqrt{1 + \lambda (1-c)}} \frac{\sqrt{n}}{\sqrt{k_2}}$, the partial derivative is always non-positive, indicating that $f(0,0,1) = 1$ is optimal. When $\rho \leq \frac{\lambda (1-c)}{\sqrt{1 + \lambda (1-c)}} \frac{\sqrt{n}}{\sqrt{k_2}}$, the derivative is first negative then positive, meaning that its optimal attains at the boundary, i.e. $\theta = 0$ or $\theta = \frac{\pi}{2}$.  If $\theta = 0$, we get a trivial result for $f(0,0,1) = 1$. If $\theta = \frac{\pi}{2}$, we have $f(\frac{\pi}{2},0,1) = \sqrt{1 + \lambda (1-c)} - \frac{\rho}{\sqrt{n}} \sqrt{k_2} $. Until now, we have shown that $\max_{\theta \in [0, \frac{\pi}{2}]} f(\theta, 0,1) = \max \left( 1, \sqrt{1 + \lambda (1-c)} - \frac{\rho}{\sqrt{n}} \sqrt{k_2} \right)$. Next, we need to compare $f(0,1,0)$ with $1$ and $\sqrt{1 + \lambda (1-c)} - \frac{\rho}{\sqrt{n}} \sqrt{k_2}$ respectively. We already have $\rho \leq \frac{\lambda c}{1 + \sqrt{1 + \lambda c}} \frac{\sqrt{n}}{\sqrt{k_1}}$, which directly imply that $f(0,1,0) \geq 1$. Moreover, since the conditions presented in Theorem~\ref{thm:strong-weak-signal} ensure $c > \frac{1}{2}$ and $k_1 < k_2$, thus
\begin{align*}
    & ~ f(0,1,0) = \sqrt{1 + \lambda c} - \frac{\rho}{\sqrt{n}} \sqrt{k_1} > \sqrt{1 + \lambda (1-c)} - \frac{\rho}{\sqrt{n}} \sqrt{k_2} = f(\frac{\pi}{2},0,1).
\end{align*}
Therefore, $\max_{\theta \in [0, \frac{\pi}{2}]} f(\theta, 1,0)\geq \max_{\theta \in [0, \frac{\pi}{2}]} f(\theta, 0,1)$\\

\noindent \textbf{Step 3} Combine the above results together, we conclude that 
\begin{enumerate}
    \item \textbf{Recovery stage.} When $\rho \in [0, ~ O(\min\{\sqrt{\lambda}, \lambda\} \sqrt{ n / k_2})]$, the problem reduces to the scenario of Theorem~\ref{thm:vector-recovery} with $f(\theta, 1, 1)$ takes the optimality. Thus, $\widehat{\bm{v}}$ recovers the ground truth $\bm{v}_*$ as described in Theorem~\ref{thm:vector-recovery}. 
    \item \textbf{Robust stage.} When $\rho \in (O(\min\{\sqrt{\lambda}, \lambda\} \sqrt{n / k_2}), ~ O(\min\{\sqrt{\lambda}, \lambda\} \sqrt{ n / k_1})]$, $f(\theta, 1, 0)$ takes the optimality, and $\widehat{\bm{v}}$ recovers the strong signal part in $\bm{v}_*$ while eliminating the weak signal. 
    \item \textbf{Overly perturbed stage.} When $\rho \in (O(\min\{\sqrt{\lambda}, \lambda\} \sqrt{ n / k_1}), ~ + \infty)$, adversarial perturbation is too large to recover enough information of the ground truth.  
\end{enumerate}

\end{proof}

\section{Numerical simulations in Section~\ref{sec:stat-results}} \label{app:numerical}

\subsection{Baselines, hardware \& software information} \label{app:baselines}

\subsubsection{Implemented MIP methods for upper (dual) bound.}

Our implemented MIP method operates in two stages to improve computational efficiency. In the first stage, given an uncorrupted sample matrix $\bm{X} \in \mathbb{R}^{n \times d}$, we reduce the original dimension $d$ of the covariance matrix $\widehat{\bm{\Sigma}}$ to a smaller dimension $\bar{d} (\geq k)$. This reduction is achieved by selecting a principal submatrix indexed by a support set $\bar{S}$ of size $\bar{d}$ such that $\bar{S}$ includes the true support $S_*$ with high probability (via projected power method). For further details, see Algorithm~\ref{alg:submatrix} for the first stage in Appendix~\ref{app:baselines}). 

\begin{algorithm}[h!]
    \caption{Submatrix Initialization (first stage)}
    \label{alg:submatrix}
    \textbf{Input:} Sample Covariance Matrix $\bm{\widehat{\Sigma}}$, initial $k$-sparse vector $\bm{v}^{(0)}$, effective dimension $\bar{d}$, total iteration time $T$
    \begin{algorithmic}[1]
    \State Set $\bm{v} = \bm{v}^{\tt target} = \bm{v}^{(0)}$
    \State Compute sparse PCA objective value $\obj = \bm{v^{\top}}\bm{\widehat{\Sigma}} \bm{v}$
    \For{$t = 1, \cdots, T$}
        \State Update $\bm{v} = \bm{\widehat{\Sigma}} \bm{v} / \|\bm{\widehat{\Sigma}} \bm{v}\|_2$
        \State Pick index set $\mathcal{I} \subseteq [d]$ wrt top $k$ largest absolute components in $\bm{v}$
        \State Compute $\lambda_{\max} (\widehat{\bm{\Sigma}}_{\mathcal{I}, \mathcal{I}})$ and corresponding eigenvector $\bm{v}_{\mathcal{I}}$
        \State Update $\bm{v}$ by setting $\bm{v}_i = [\bm{v}_{\mathcal{I}}]_i$ for $i \in \mathcal{I}$, and zero otherwise        
        \If{$\lambda_{\max} (\widehat{\bm{\Sigma}}_{\mathcal{I}, \mathcal{I}}) > \obj$}
            \State Update $\obj = \lambda_{\max} (\widehat{\bm{\Sigma}}_{\mathcal{I}, \mathcal{I}})$ and $\bm{v}^{\tt target} = \bm{v}$
        \EndIf
    \EndFor
        \State Set the support of $\bm{v}^{\tt target}$ as $S^{\tt{target}}$
        \State Add another $\bar{d} - k$ indices with the largest diagonal elements to the current support $S^{\tt{target}}$ and obtain the effective support $\bar{S}$
        \State Obtain submatrix $\bm{\widehat{\Sigma}^{\tt{sub}}} = \bm{\widehat{\Sigma}}_{\bar{S},\bar{S}}$
    \end{algorithmic}
    \textbf{Output:} $\bm{\widehat{\Sigma}^{\tt{sub}}}$. 
    \end{algorithm}

In particular, given an initial point $\bm{v}_0$, Algorithm~\ref{alg:submatrix} first selects a support set $S^{\tt{target}}$ of size $k$ by computing the sparse principal component of $\widehat{\bm{\Sigma}}$ via the projected power method. Under the Assumption~\ref{assump:sample-model}, such support set $S^{\tt{target}}$ coincide with the true support set $S_*$ with high probability. The rest $\bar{d} - k$ indices are chosen by picking the largest $\bar{d} - k$ diagonal elements' indices within the rest index set, i.e.,  $[d] \backslash S^{\tt{target}}$, which shares the similar insights as diagonal thresholding method for sparse PCA problem. 
Then, merging these $\bar{d} - k$ indices into the first selected support set $S$ results in our effective support set $\bar{S}$ of size $\bar{d}$. 

Based on the first stage, the second stage of our proposed MIP implements proposed MIP formulation (i.e., \eqref{eq:featurewise-MIP} and \eqref{eq:featurewise-MIP-r}) on restricted  covariance matrix $\bm{\widehat{\Sigma}^{\tt{sub}}}$, i.e., the submatrix of $\bm{\widehat{\Sigma}}$ obtained by picking rows and columns indexed by $\bar{S}$. 
    
To be clear, we use $\bm{v}^{\MIP / \MIPr}, \opt^{\MIP / \MIPr}, \ub^{\MIP / \MIPr}$ to denote the optimal solution (or current best feasible solution), optimal value (or current best objective value), and corresponding upper (dual) bound obtained from solving Fomulation~\eqref{eq:featurewise-MIP} or Fomulation~\eqref{eq:featurewise-MIP-r}, respectively.

\subsubsection{Dual Baseline: Sparse PCA.}
Easy to observe that the vanilla sparse PCA problem 
\begin{align*}
    \max_{\bm{v} \in \mathcal{V}_k} ~ \bm{v}^{\top} \widehat{\bm{\Sigma}} \bm{v}
\end{align*}
could always be an upper bound for the optimal value of featurewise adversarial perturbed sparse PCA. We use $\bm{v}^{\spca}, \opt^{\spca}, \ub^{\spca}$ to denote its corresponding optimal solution (or current best feasible solution), optimal value, and upper (dual) bound. 

\subsubsection{Primal Baseline: Projected Power Method.} We take the vanilla projected power method (\PPM) as our primal baseline, as presented in Algorithm~\ref{alg:PPM}. 
\begin{algorithm}[h]
    \caption{\PPM} \label{alg:PPM} 
    \textbf{Input:} Initial point $\bm{v}^{(1)}$, iteration bound $T$, objective function $\mathcal{F}$ of Formulation~\eqref{eq:emp-adv-spca-2}
    \begin{algorithmic}[1]
    \For{$t = 1, \ldots, T$ and early stop criteria does not meet}
    \State Compute  $\nabla_{\bm{v}} \mathcal{F} (\bm{v}^{(t)})$
    \State Project to feasible set $\bm{v}^{(t + 1)} = \proj_{\mathcal{V}_k} (\nabla_{\bm{v}} \mathcal{F} (\bm{v}^{(t)})) $
    \EndFor
    \end{algorithmic}
    \textbf{Output:} $\bm{v}^{(T)}$. 
\end{algorithm} 
Additionally, in order to accelerate computation, we check the difference of vectors from consecutive iterations during each iteration. If their $\ell_2$ norm is smaller than the pre-determined threshold $\epsilon = 10^{-6}$, we early stop the iteration. Still, we use $\bm{v}^{\PPM} \in \mathcal{V}_k, \obj^{\PPM}$ to denote the primal feasible solution, the objective value obtained by \PPM.

\subsubsection{Hardware \& Software.} All experiments are conducted in MacBook Air with an Apple M2 Silicon (3.49 GHz, 2.42 GHz) and 16GB Memory (128-bit LPDDR5 3200 MHz). The proposed method and other methods are solved using Gurobi 11.0.0 in Python 3.11.5.

\subsection{Additional numerical simulations} \label{app:additional-numerical}

This subsection reports additional numerical simulation results (see Figure~\ref{figure:syn-data-1} below) as mentioned in Section~\ref{sec:numerical}. 

\begin{figure}[!h]
\centering
\begin{subfigure}
\centering
\includegraphics[width=0.31\linewidth]{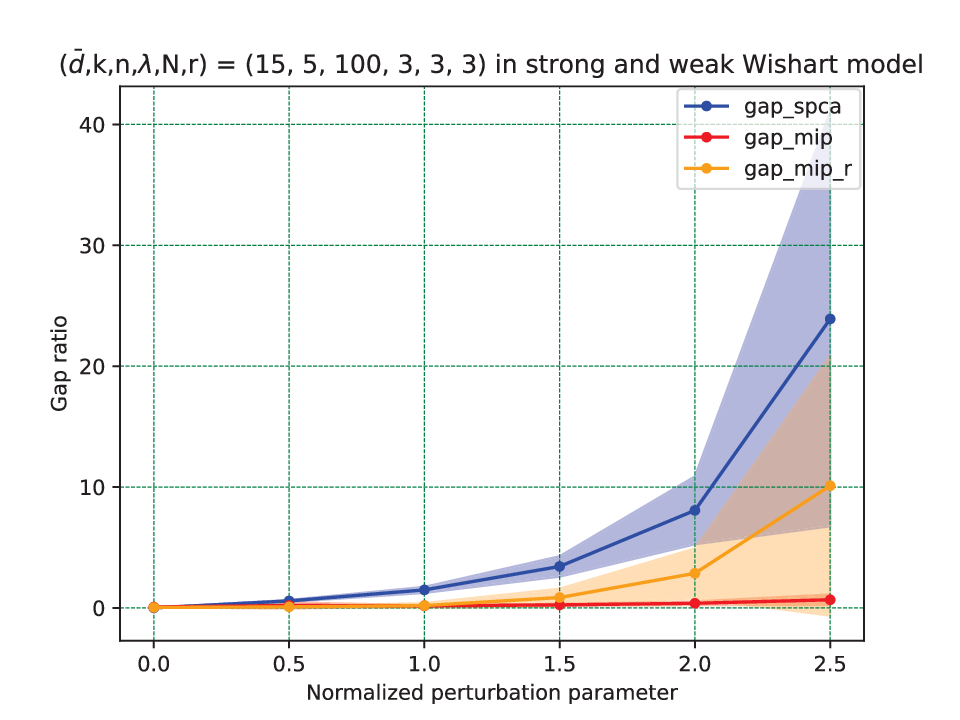}
\end{subfigure}
\hfill
\begin{subfigure}
\centering
\includegraphics[width=0.31\linewidth]
{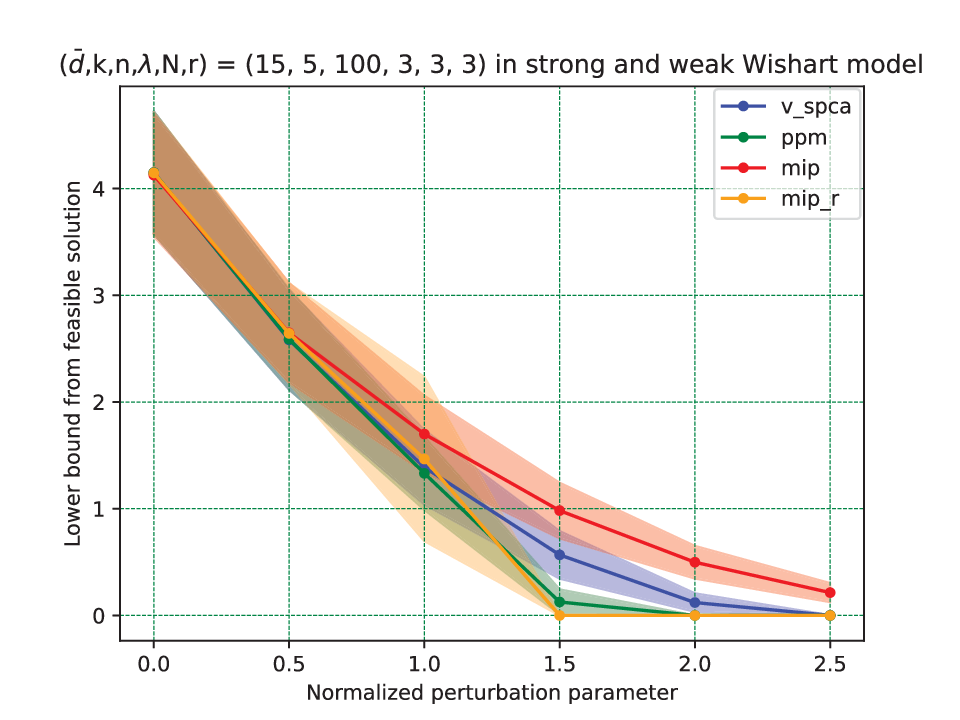}
\end{subfigure}
\hfill
\begin{subfigure}
\centering
\includegraphics[width=0.31\linewidth]
{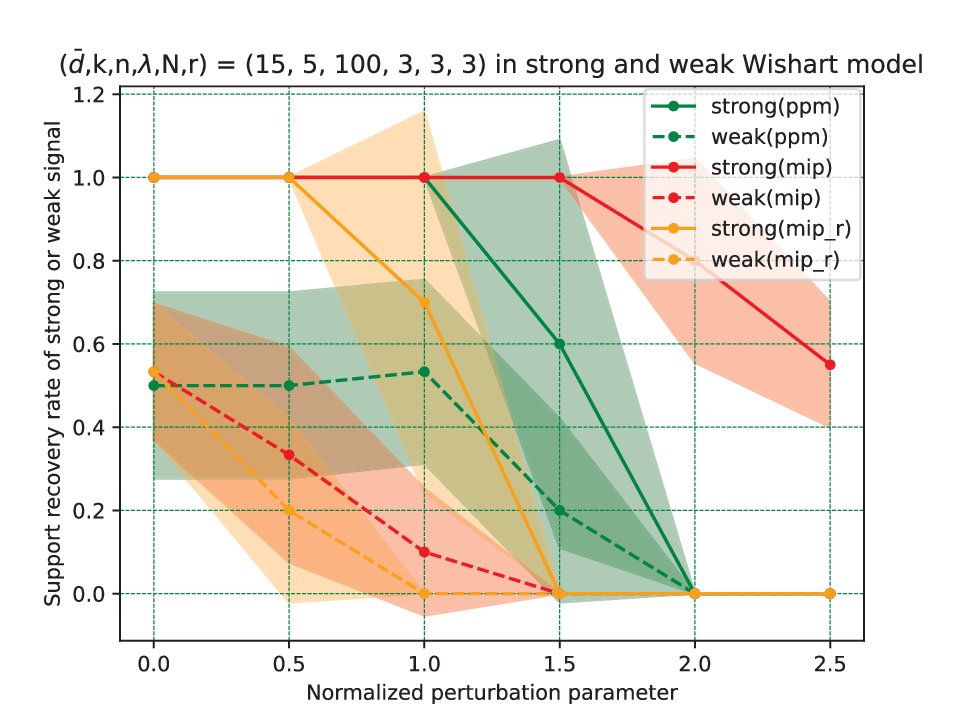}
\end{subfigure}

\centering
\begin{subfigure}
\centering
\includegraphics[width=0.31\linewidth]{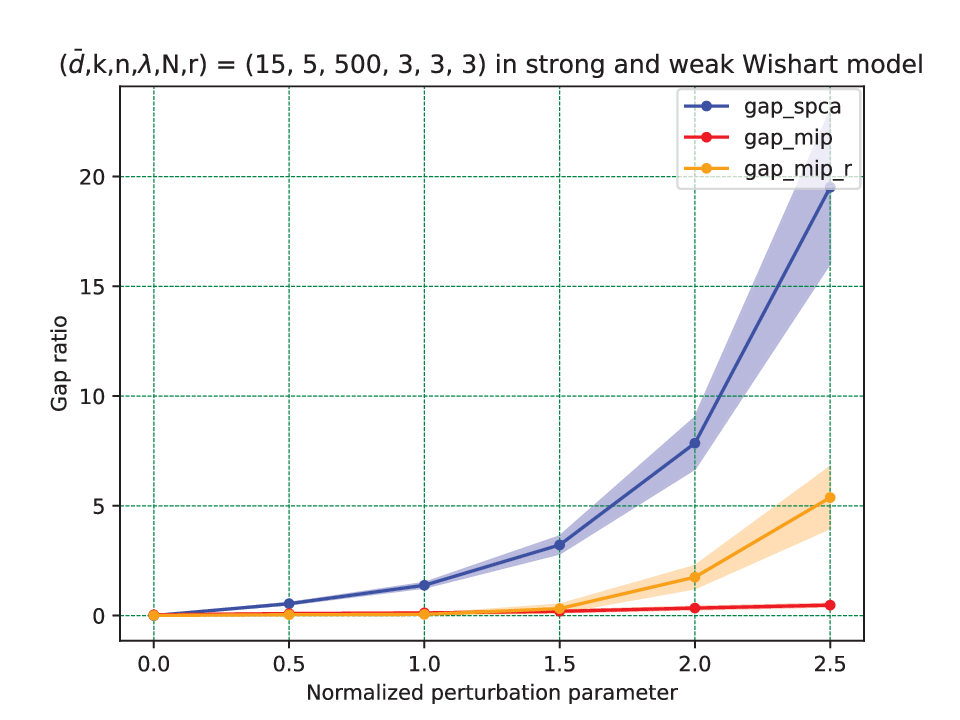}
\end{subfigure}
\hfill
\begin{subfigure}
\centering
\includegraphics[width=0.31\linewidth]
{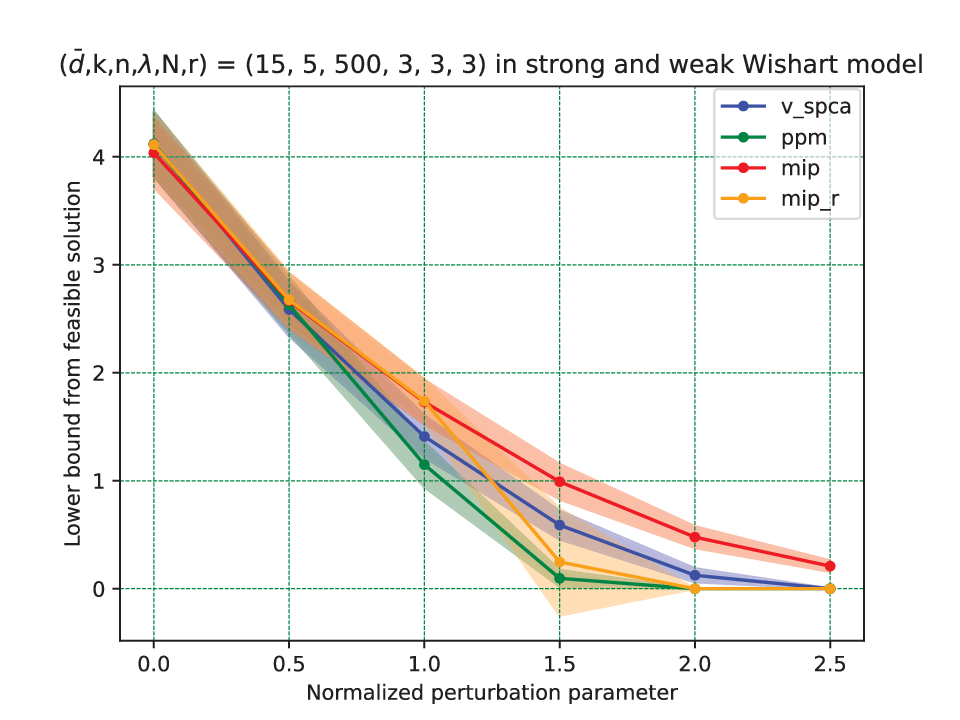}
\end{subfigure}
\hfill
\begin{subfigure}
\centering
\includegraphics[width=0.31\linewidth]
{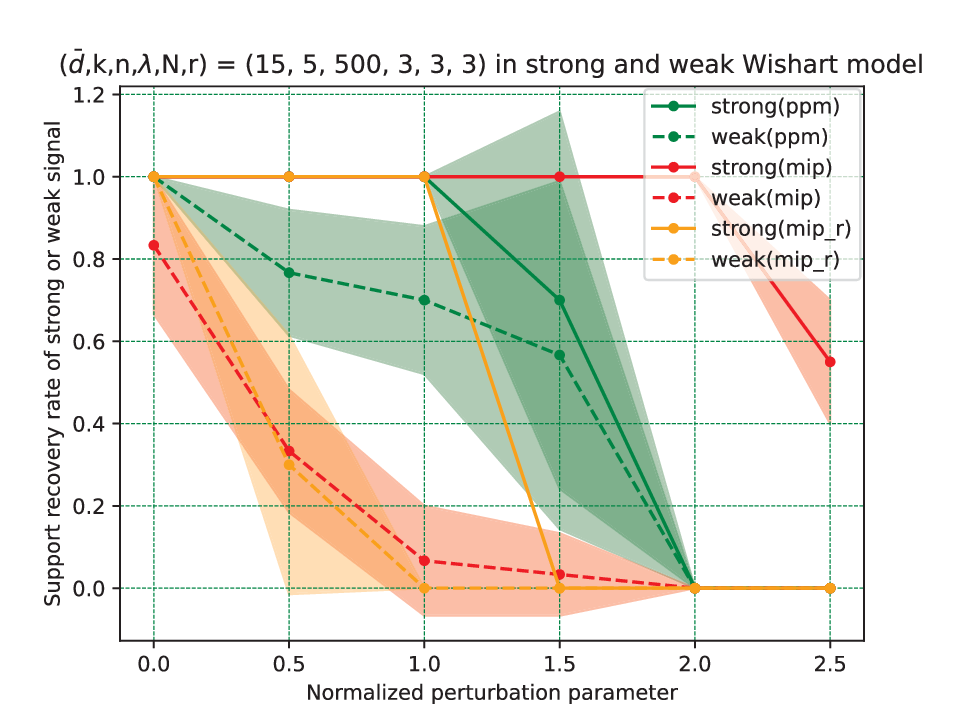}
\end{subfigure}

\caption{\textbf{Numerical simulations on strong and weak signals.} 
This figure compares the gap ratio, objective function value, and support recovery rate v.s. the normalized perturbation parameter $\bar{\rho}$ of \PPM \ and two proposed MIP formulations under two different parameter settings as mentioned in Section~\ref{sec:numerical}. The first and second rows consist of panels for $n = 100$ and $n = 500$, respectively. \\\\
In the first column, we plot the gap ratio (\gap) v.s. the normalized perturbation parameter $\bar{\rho}$. The
three solid curves in each panel correspond to the averaged values over 10 independent trials of corresponding methods; the shaded parts represent the empirical standard deviations over 10 independent trials. Smaller \gap \ means that the upper bound is tighter. It is easy to observe that the proposed \MIP \ method outperforms the other two methods, which validates our theoretical analysis. The \gap \ of the proposed \MIP \ method is close to 0, indicating that the MIP formulation is a very tight upper bound for the problem. \\\\
The second column plots the objective function values (\obj) of \eqref{eq:emp-adv-spca-2} v.s. the normalized perturbation parameter $\bar{\rho}$. Similarly, the four solid curves in each panel correspond to the averaged values over 10 independent trials of corresponding methods; and their shaded parts represent the empirical standard deviations over 10 independent trials. As we can observe, the proposed MIP method performs best when the perturbation is large, indicating that it is a more resilient method for finding the optimal solution.\\\\
The third column plots the support recovery rate v.s. the normalized perturbation parameter $\bar{\rho}$. In each panel, the three solid curves represent the averaged \emph{strong signal} recovery rate over 10 independent trials for each method, while the three dashed curves depict the averaged \emph{weak signal} recovery rate over 10 independent trials for each method. The shaded areas indicate empirical standard deviations across these trials. Our findings show that both the strong and weak signal recovery rates of \PPM \ consistently outperform the two proposed methods, suggesting that \PPM \ could be an effective approach for recovering the ground truth in this problem setting. A likely explanation for \PPM's strong performance lies in its strategy of retaining the $k$ largest indices of the gradient at each iteration. Given that the gradient of \eqref{eq:emp-adv-spca-2} is $2 \left(\|\bm{X} \bm{v}\|_2 - \rho \|\bm{v}\|_1 \right) \left(\frac{\bm{X}^{\top} \bm{X} \bm{v}}{\|\bm{X} \bm{v}\|_2} - \rho \tt{sgn} \left(\bm{v} \right) \right)$, the term $\rho \tt{sgn} \left(\bm{v} \right)$ becomes dominant when $\rho$ is large. Consequently, the support at each iteration tends to align with that of the previous iteration, often leading to a final computed support equal to the initial support, i.e., the support of $\bm{v}^{\spca}$, which corresponds to the true support with high probability. This observation explains why \PPM \ can retain the true support even under high perturbations.
}
\label{figure:syn-data-1}
\end{figure}

\end{document}